\documentclass[10pt]{amsart}
\usepackage[T2A]{fontenc}
\usepackage[utf8]{inputenc}
\usepackage[unicode]{hyperref}
\usepackage{color}
\usepackage{amsfonts,amssymb,amsmath,amscd,amsthm}
\usepackage[dvips]{graphicx}
\usepackage[matrix,arrow,curve]{xy}

\hypersetup{
  colorlinks   = true,
  urlcolor     = red,
  linkcolor    = black,
  citecolor   = red
}

\makeatletter
\def\@settitle{\begin{center}%
    \baselineskip14\p@\relax
    \bfseries
    \@title
  \end{center}%
} \makeatother

\usepackage{longtable}

\makeatletter\@addtoreset{equation}{section}\makeatother
\renewcommand{\theequation}{\thesection.\arabic{equation}}

\renewcommand{\thesubsection}{\bf\thesection.\arabic{equation}}
\makeatletter\@addtoreset{subsection}{equation}\makeatother

\newtheorem{theorem}[equation]{Theorem}
\newtheorem*{theorem*}{Theorem}
\newtheorem{proposition}[equation]{Proposition}
\newtheorem*{proposition*}{Proposition}

\newtheorem{lemma}[equation]{Lemma}
\newtheorem*{lemma*}{Lemma}
\newtheorem{corollary}[equation]{Corollary}
\newtheorem*{corollary*}{Corollary}
\newtheorem{conj}[equation]{Conjecture}
\newtheorem{prop}[equation]{Proposition\,-\,definition}

\theoremstyle{definition}
\newtheorem{example}[equation]{Example}

\theoremstyle{remark}
\newtheorem{remark}[equation]{Remark}
\makeatletter
\newcommand{\symbitem}[1]{\item[#1]%
\renewcommand{\@currentlabel}{#1}\ignorespaces}
\makeatother
\newcommand{\beq}{\begin{equation}}
\newcommand{\eeq}{\end{equation}}
\newcommand{\beqa}{\begin{eqnarray}}
\newcommand{\eeqa}{\end{eqnarray}}
\newcommand{\beaa}{\begin{eqnarray*}}
\newcommand{\ben}{\begin{eqnarray*}}
\newcommand{\eaa}{\end{eqnarray*}}
\newcommand{\een}{\end{eqnarray*}}

\def \E {\mathcal{E}}

\def \L {\mathcal{L}}

\def\OO {\mathcal{O}}

\def \V {\mathcal{V}}

\def \C {\mathbb{C}}
\def \F {\mathbb{F}}

\def \P {\mathbb{P}}
\def \Q {\mathbb{Q}}
\def \R {\mathbb{R}}

\def \Z {\mathbb{Z}}

\def \ge {\geqslant}

\def \le {\leqslant}

\def \kappa {\varkappa}

\def\={\;=\;}
\def\bal{\begin{aligned}}
\def\eal{\end{aligned}}

\newcommand{\lra}{\longrightarrow}

\newcommand{\Spec}{{\text{Spec }}}

\newcommand{\udot}{{\:\raisebox{3pt}{\text{\circle*{1.5}}}}}
\def \bullet {\udot}

 \DeclareMathOperator{\Hom}{Hom}
 \DeclareMathOperator{\Aut}{Aut}
 
 \DeclareMathOperator{\Ext}{Ext}
 
 \DeclareMathOperator{\mult}{mult}

\DeclareMathOperator{\Sing}{Sing} 
\DeclareMathOperator{\Pic}{Pic}
\DeclareMathOperator{\jac}{Jac} \DeclareMathOperator{\slg}{SL}

\providecommand{\arxiv}[1]{\href{http://arxiv.org/abs/#1}{arXiv:#1}}

\newtheorem*{notation}{Notation}

\newtheorem*{assumption}{Assumption}

\thanks{{\it MS 2010 classification}: 14J27, 18E30, 14J81}

\thanks{{\it Key words}: elliptic surface, derived category, phantom}

\title{\Large Exceptional collections and phantoms of special Dolgachev surfaces}

\author{Ilya Karzhemanov}
\address{\newline{\normalsize Laboratory of AGHA, Moscow Institute of Physics and Technology, 9 Institutskiy per., Dolgoprudny,
Moscow Region, 141701, Russia}
\medskip
\newline{\normalsize
National Research University Higher School of Economics,
Faculty of Mathematics, NRU HSE, 6 Usacheva str., Moscow,
119048, Russia}
\smallskip
\newline{\it E-mail address}: karzhemanov.iv@mipt.ru}

\author{Ludmil Katzarkov}
\address{\newline{\normalsize University of Miami, Coral Gables, FL}
\medskip
\newline{\normalsize
Institute of Mathematics and Informatics, Bulgarian Academy of
Sciences, Acad. Georgi Bonchev Str., Block 8, Sofia, 1113,
Bulgaria}
\medskip
\newline{\normalsize
National Research University Higher School of Economics,
Laboratory of Mirror Symmetry, NRU HSE, 6 Usacheva str., Moscow,
119048, Russia}
\smallskip
\newline{\it E-mail address}: lkatzarkov@gmail.com}

\begin{document}

\begin{abstract}
We provide an explicit description of exceptional collection of
maximal length in the derived category $D^b(Y)$ for a particular
class of elliptic surfaces $Y$. The existence of a non\,-\,trivial
semiorthogonal complement (a ``\,phantom\,'') of this collection
is also established.
\end{abstract}

\maketitle

\bigskip

\section{Introduction}
\label{section:int}

Let $Y$ be a smooth projective surface over $\C$.\footnote{~We
will use standard notions and results (although we recall some of
them in the text) from \cite{kodaira}, \cite{bpv}, \cite{bo-yu},
concerning compact complex surfaces.} Assume that $q(Y) = p_g(Y) =
0$ (i.\,e. irregularity and geometric genus of $Y$ are trivial).
Consider the derived category $D^b(Y)$ of coherent sheaves on $Y$.
Then Theorem 6 in \cite{vial} asserts that any numerically
exceptional collection (if exists) $A_1,\ldots,A_N$ of maximal
length $N = c_2(Y)$ of objects $A_i \in D^b(Y)$ (see
{\ref{subsection:pre-2}} and {\ref{subsection:pre-3}} below for a
setup) yields a similar collection of \emph{line bundles} (again
denoted $A_i$) on $Y$, so that the Chow lattice $N^1(Y)$ of
$1$\,-\,cycles on $Y$ is generated by $A_i$ and is unimodular with
respect to the intersection pairing. The aim of the present paper
is to apply this fact in order to
$$
\begin{array}{c}
\textit{find \emph{explicit} exceptional collections of maximal
length on certain (minimal) elliptic surfaces $Y$.}
\end{array}
$$
Our candidate for $Y$, allowing an algebro\,-\,geometric
construction, will be members of a family of the so\,-\,called
\emph{Dolgachev surfaces} (see {\ref{subsection:pre-1}}). In this
case, direct calculation yields a numerically exceptional
collection of maximal length on $Y$ (see {\ref{subsection:pre-2}}
and Theorem~\ref{theorem:exc-collection-12}), and an additional
geometric reasoning shows that the given collection is actually
exceptional (see {\ref{subsection:pre-3}}). We apply these
constructions in order to find a \emph{phantom category} in
$D^b(Y)$ (see Section~\ref{section:pha}). To our best knowledge
this is the most explicit description of the phantom phenomenon
for $D^b$ of algebraic surfaces having positive Kodaira dimension
(compare with \cite{g-s}, \cite{gkms}, \cite{gorch-orlov},
\cite{chi-lee} and \cite{tevelev-urzua}).

Section~\ref{section:log-trans} generalizes the preceding
algebro\,-\,geometric considerations and treats the problem of
maximal length exceptional collections on Dolgachev surfaces
$\frak{S}$ in \emph{transcendental} terms
(Corollary~\ref{theorem:log-trans-exc-cor} together with
\cite[Lemma 4.3]{bridgeland-maciocia} also prove the existence of
phantoms in $D^b(\frak{S})$ for a particular class of $\frak{S}$
related to $Y$). Appendix at the end contains a non\,-\,technical
discussion of relations with other work and some further
generalizations (yet this section is a bit ``\,heavy\,'' for
Introduction and an interested reader may turn to it, if need be,
after reading the rest of the paper).

\bigskip

\section{Explicit constructions}
\label{section:pre}

\refstepcounter{equation}
\subsection{Dolgachev's construction}
\label{subsection:pre-1}

Firstly, one starts with two lines $L_1,L_2\subset\P^2$ and a
plane quartic curve $Q$ (to be specified). In projective
coordinates $[x:y:z]$, let the equations of $L_1$ and $L_2$ be $z
= 0$ and $y = 0$, respectively. Further, one considers the points
\[q_1 := [1:0:0],\qquad q_5 := [0:0:1],\qquad q_9 := [1:0:1]\] and
some additional points
\[q_1 > q_2 > q_3 > q_4,\qquad q_5 > q_6 > q_7 > q_8,\qquad q_9,\] where ``\,$q_{i} > q_{i+1}$\,'' signifies that
the point $q_{i+1}$ is infinitely close to $q_i$. Then $Q$ is
chosen to have $q_5,q_6,q_7$ as double points and pass through all
other $q_i$ (see Figure~\ref{fig-1}). Furthermore, $L_1$ passes
through $q_1,q_2,q_3$, while $L_2$ passes through $q_1,q_5,q_9$.
The equation of $Q$ can be chosen in the form
\[(xz + y^2)^2 - x^3z + ax^2yz + axy^3 = 0\]
for any $a\in\C$. Finally, one picks a smooth cubic curve $C$,
which passes through all $q_i$, with equation being
\[xz^2 - x^2z + a'xyz + y^2z + ay^3 = 0,\]
where $a,a'$ satisfy the relation $(a' - a)^2 - a(a' - a) + 1 =
0$.

\begin{remark}
\label{remark:exp-constr} The point $q_6$ corresponds to the
tangent line $(x = 0)$ of $Q$ and $L_1$ gives a flex point on $Q$.
Take $x$ and $y$ for local coordinates near $q_5$. Then we have $q_6 = (0,0)$ in the blow\,-\,up chart $(x_1,y_1) \mapsto (x_1y_1 = x,y_1 = y)$ of $q_5$. Likewise, $q_7$ corresponds to $(0,0)$ in the blow\,-\,up chart $(x_2,y_2) \mapsto (x_2y_2 = x_1,y_2 = y_1)$ of $q_6$, and the relation between the above parameters $a,a'$ means that the proper transforms of $Q$ and $C$ under the iterated blow\,-\,ups pass through $q_8$. Similar description applies to the points $q_1,q_2,q_3$ and $q_4$.
\end{remark}

\begin{figure}[h]
\includegraphics[scale=1.3]{pi.mp.1}
\caption{~}\label{fig-1}
\end{figure}

From the data just presented one constructs a (Halphen) pencil
\beq\label{igor-pen}\lambda yz((xz + y^2)^2 - x^3z + ax^2yz +
axy^3) + \mu(xz^2 - x^2z + a'xyz + y^2z + ay^3)^2 = 0,\qquad
\lambda,\mu\in\C,\eeq of plane curves. Let $X$ be the surface
obtained by blowing up $\P^2$ at nine points $q_i$. Then
\eqref{igor-pen} yields an elliptic fibration
$p_{\scriptscriptstyle X}: X \lra \P^1$. It has unique multiple
fiber $2C$ (we identify the cubic $C \subset \P^2$ with its proper
transform on $X$). Indeed, if there were another fiber $2C'$ of $p_{\scriptscriptstyle X}$, then \eqref{igor-pen} would coincide with the pencil spanned by the cubics $C$ and $C'$, a contradiction (note also that $p_{\scriptscriptstyle X}$ does not have a triple fiber by B\'ezout). 

Further, if $R_1,\ldots,R_9$ are the exceptional
curves of the contraction $X \lra \P^2$, then (the proper
transforms of) $Q,L_1$ and $L_2$ together with the \emph{$(-2)$\,-\,curves} $R_1,R_2,R_3,R_5,R_6,R_7$ form a singular fiber of type $I_9$ (cf.
\cite[V.7, Table 3]{bpv} and Figure~\ref{fig-2} below). Let us
determine the rest of the singular fibers of
$p_{\scriptscriptstyle X}$.

Denote by $p: \jac X\lra\P^1$ the Jacobian fibration associated
with $p_{\scriptscriptstyle X}$.

\begin{lemma}
\label{theorem:sing-fibs-jac-x} $p$ has at least $4$ singular
fibers.
\end{lemma}

\begin{proof}
Suppose first that $p$ has at most $2$ singular fibers. In this
case, the global monodromy $\Gamma := \Gamma_{\jac X} \subset
\slg_2(\Z)$ (cf. {\ref{subsection:pha-1}} below) is isomorphic to
$\Z$, which implies that the index $[\slg_2(\Z):\Gamma]$ is
infinite. However, since $p$ is (algebraic and) Jacobian, one has
$[\slg_2(\Z):\Gamma] < \infty$ (see e.\,g. \cite[Section
3]{bo-yu}), a contradiction.

Further, if $F_1,F_2,F_3$ are the only singular fibers of $p$,
then let $g_i \in \slg_2(\Z)$ be the images of loops around
$p(F_i)$. Now, since the fundamental group $\pi_1(\P^1 \setminus
\{p(F_3)\})$ is trivial, it follows that $\Gamma$ is generated by
$g_1,g_2$ subject to the relation $g_1 \cdot g_2 = 1$, which is
again impossible.
\end{proof}\\

It follows from Lemma~\ref{theorem:sing-fibs-jac-x} that the
remaining singular fibers of $p_{\scriptscriptstyle X}$ are all of
type $I_1$ (cf. Figure~\ref{fig-2}). Indeed, logarithmic
transformation from $\jac X$ to $X$ induces an analytic
isomorphism away from the unique multiple fiber on $X$ (see
\cite[V.13]{bpv}), where the latter fiber is supported on a smooth
curve isomorphic to $C$ (cf. \eqref{igor-pen}). Then from $c_2(X)
= c_2(\jac X) = 12$ we obtain
$$
12 = c_2(X) = \sum_{b\in\P^1\, :\, p^{-1}(p) =
p_{\scriptscriptstyle X}^{-1}(b) := F_b\,\text{is
singular}}\chi_{\text{top}}(F_b),
$$
with the fiber $I_9$ having $\chi_{\text{top}} = 9$ and
$\chi_{\text{top}}(F_b) > 0$ for at least \emph{three more} $b$,
so that all the corresponding $F_b$ are of type $I_1$ and one gets
exactly three of them.

\begin{figure}[h]
\includegraphics[scale=1.5]{pii.1}
\caption{~}\label{fig-2}
\end{figure}

Fix some general fiber $F$ of $p_{\scriptscriptstyle X}$. One
computes the divisor
\[D := F + (2L_1 + R_3) + (2R_1 + L_2) + (2R_6 + R_7)\]
is divisible by $3$ in $\Pic X$. Then we can form a $3:1$ cyclic
covering $X' \lra X$ with $D$ as its ramification locus.
Furthermore, the curves $L_1,R_3,\ldots,R_7$ in $D$ can be
contracted by some $\phi: X \lra W$ to three
$\displaystyle\frac{1}{3}(1,2)$\,-\,points $O_1,O_2,O_3 \in
W$.\footnote{~This is standard via the construction of a nef and
big divisor $\mathcal{L}$ on $S$ satisfying $\mathcal{L} \cdot L_1
= \mathcal{L} \cdot R_3 = \ldots = \mathcal{L} \cdot R_7 = 0$, so
that the linear system $|m\mathcal{L}|$ is basepoint\,-\,free, $m
\gg 1$ (compare with \cite[Section 3]{keum}).} The surface $W$ is
algebraic and $X' \lra X$ descends to a $3:1$ cyclic covering
$\psi: Y \lra W$ ramified at $O_i$ and $F$ (we denote every cycle
$\psi^*\phi_*Z$ on $Y$, where $Z\in\Pic X$, again by $Z$). Namely,
we have $X' = \Spec_X\displaystyle\bigoplus_{i=0}^2\OO_X(-iD)$,
whereas $Y =
\Spec_W\displaystyle\bigoplus_{i=0}^2\phi_*\OO_X(-iD)$ for the
coherent sheaves $\phi_*\OO_X(-iD)$ on $W$ (cf.
Figure~\ref{fig-3}).

\bigskip

\begin{figure}[h]
\includegraphics[scale=0.9]{pi1.mp.1}
\caption{~}\label{fig-3}
\end{figure}

\refstepcounter{equation}
\subsection{The Picard lattice of $Y$}
\label{subsection:pre-2}

Let $\mu_3$ be the Galois group of the above covering $\psi: Y
\lra W$. Note that the surface $Y$ is smooth because
$\phi_*F\cap\Sing W = \emptyset$ and locally analytically over
every $O_i \in W$, for $\Sing W = \{O_1,O_2,O_3\}$, morphism
$\psi$ coincides with the quotient $\C^2 \lra \C^2 \slash \mu_3$.
Construction in {\ref{subsection:pre-1}} also provides an evident
elliptic fibration $p_{\scriptscriptstyle Y}: Y \lra \P^1$, with
fibers $\Phi_1,\ldots,\Phi_4$ of type $I_3$ and two multiple
fibers supported on the smooth elliptic curves $C$ and $F$,
respectively, such that $2C \equiv 3F \equiv \Phi_i$ (numerically
on $Y$) --- these are the only singular fibers of
$p_{\scriptscriptstyle Y}$ (cf.
Lemma~\ref{theorem:sing-fibs-jac-x}). Recall that
\beq\label{can-bun} K_Y = 2F - C \equiv \frac{1}{6}\,\Phi_i \eeq
(see e.\,g. \cite[V.12]{bpv}).

\begin{proposition}
\label{theorem:curves-on-y} Let $R_2,R_5,\ldots,G_2$ be the curves
as in Figure~\ref{fig-3}. Then their intersection data is given by
Table A below.
\end{proposition}

\begin{proof}
Note that the cycle
\[\gamma := R_8 + \frac{2}{3}\,R_7 + \frac{1}{3}\,R_6\]
on $X$ is relatively trivial with respect to $\phi: X \lra W$ because of the following identities (cf. Figure~\ref{fig-2}):
$$
L_2 \cdot R_8 = L_2 \cdot
R_7 = L_2 \cdot R_6 = 0,\ \text{and similarly for the curves}\ L_1, R_1, R_3,\ \text{which gives}
$$
$$
\gamma \cdot L_2 = \gamma \cdot L_1 = \gamma \cdot R_1 = \gamma \cdot R_3 = 0;
$$
$$
R_6 \cdot R_8 = 0,\ R_6 \cdot R_7 = R_8 \cdot R_7 = 1,\ (R_6^2) = (R_7^2) = -2, \ \text{which gives} 
$$
$$
\gamma \cdot R_6 = \gamma \cdot R_7 = 0.
$$
Then we have
\[(R_8^2) := 3(\gamma^2) = -1\]
on $Y$ by definition. It also follows from \eqref{can-bun} that
\[K_Y \cdot R_8 := 3(\frac{1}{6}\,\Phi_i \cdot \gamma) = 1.\]
In particular, the curve $R_8 \subset Y$ is irreducible, with \emph{non\,-\,trivial} action of the group $\mu_3$.

Similarly, the cycle \[\delta :=R_5 + \frac{2}{3}\,L_2 +
\frac{1}{3}\,R_1\] is trivial over $W$, which gives
\[R_5 \cdot R_8 := 3(\delta \cdot \gamma) = 1\]
on $Y$. We also get $R_8 \cdot H_1 = R_8 \cdot H_2 = 2$ because
$R_8$ is $\mu_3$\,-\,invariant and $\mu_3$ permutes $H_i$ in the fiber $\Phi_2$. Other
intersections are treated in exactly the same way and we skip
their computation.
\end{proof}\\

\begin{table}[h]

\begin{center}

\begin{tabular}[c]{c|c|c|c|c|c|c|c|c|c|c|}
& $K_Y$ & $R_2$ & $R_5$ & $R_8$ & $H_1$ & $H_2$ & $F_1$ & $F_2$ & $G_1$ & $G_2$\\
\hline $K_Y$   & 0 & 0 & 0 & 1 & 0 & 0 & 0 & 0 & 0 & 0 \\
\hline $R_2$ & 0 & -2 & 1 & 0 & 0 & 0 & 0 & 0 & 0 & 0 \\
\hline $R_5$ & 0 & 1 & -2 & 1 & 0 & 0 & 0 & 0 & 0 & 0 \\
\hline $R_8$ & 1 & 0 & 1 & -1 & 2 & 2 & 2 & 2 & 2 & 2 \\
\hline $H_1$ & 0 & 0 & 0 & 2 & -2 & 1 & 0 & 0 & 0 & 0 \\
\hline $H_2$ & 0 & 0 & 0 & 2 & 1 & -2 & 0 & 0 & 0 & 0 \\
\hline $F_1$ & 0 & 0 & 0 & 2 & 0 & 0 & -2 & 1 & 0 & 0 \\
\hline $F_2$ & 0 & 0 & 0 & 2 & 0 & 0 & 1 & -2 & 0 & 0 \\
\hline $G_1$ & 0 & 0 & 0 & 2 & 0 & 0 & 0 & 0 & -2 & 1 \\
\hline $G_2$ & 0 & 0 & 0 & 2 & 0 & 0 & 0 & 0 & 1 & -2\\
\hline

\end{tabular}

\end{center}

\bigskip
\text{Table {\rm A}}

\end{table}

\bigskip

\begin{notation}
In what follows, for notation's simplicity, we will be usually
neglecting the use of $\equiv$ , thus identifying any element from
$\Pic Y$ with its class in $H^2(Y,\Z)$. However, in some cases
with a possibility of confusion, we will distinguish between
$\equiv$ and $=$. We will also specify sometimes the linear
equivalence $\sim$ between Cartier divisors.
\end{notation}

\begin{lemma}
\label{theorem:3-torsion-classes} The classes
\[
\begin{array}{ccc}
\bigskip
K_Y,\ R_2,\ R_5,\ R_8,\ H_1,\ H_2,\ F_1,-\frac{1}{3}\,K_Y -
\frac{1}{3}\,R_2 + \frac{1}{3}\,R_5 - \frac{1}{3}\,F_1 +
\frac{1}{3}\,F_2 -
\frac{1}{3}\,H_1 + \frac{1}{3}\,H_2,\\
\bigskip
-\frac{1}{3}\,K_Y - \frac{1}{3}\,R_2 + \frac{1}{3}\,R_5 +
\frac{1}{3}\,F_1 - \frac{1}{3}\,F_2 + \frac{1}{3}\,G_1 -
\frac{1}{3}\,G_2, \ G_2
\end{array}
\]
form a basis of the lattice $\Pic Y$.
\end{lemma}

\begin{proof}
Note that the collection $\left\{K_Y,R_2,R_5,R_8,H_1,H_2,F_1,F_2,G_1,G_2\right\}$ constitutes a basis of the ($10$\,-\,dimensional) linear space $\Pic
Y \otimes \R$ because its Gramm matrix $A$, given by Table A, has determinant $-81$. In turn,
the matrix $M$, transforming this basis to the one proposed by
lemma, is
\[\left(%
\begin{array}{cccccccccc}
  1 & 0 & 0 & 0 & 0 & 0 & 0 & -\frac{1}{3} & -\frac{1}{3} & 0 \\
  0 & 1 & 0 & 0 & 0 & 0 & 0 & -\frac{1}{3} & -\frac{1}{3} & 0 \\
  0 & 0 & 1 & 0 & 0 & 0 & 0 & \frac{1}{3} & \frac{1}{3} & 0 \\
  0 & 0 & 0 & 1 & 0 & 0 & 0 & 0 & 0 & 0 \\
  0 & 0 & 0 & 0 & 1 & 0 & 0 & -\frac{1}{3} & 0 & 0 \\
  0 & 0 & 0 & 0 & 0 & 1 & 0 & \frac{1}{3} & 0 & 0 \\
  0 & 0 & 0 & 0 & 0 & 0 & 1 & -\frac{1}{3} & \frac{1}{3} & 0 \\
  0 & 0 & 0 & 0 & 0 & 0 & 0 & \frac{1}{3} & -\frac{1}{3} & 0 \\
  0 & 0 & 0 & 0 & 0 & 0 & 0 & 0 & \frac{1}{3} & 0 \\
  0 & 0 & 0 & 0 & 0 & 0 & 0 & 0 & -\frac{1}{3} & 1 \\
\end{array}%
\right)\] and has $\det M = 1/9$.\footnote{~All linear algebra
computations in the paper are carried via PARI/GP.} Then it is
immediate to check that $M^tAM \in \text{SL}_{10}(\Z)$ for the
transposed matrix $M^t$ of $M$. Let $Q$ be the quadratic form, corresponding to the matrix $M^tAM$, and $\Lambda \subset \Pic
Y \otimes \R$ be the lattice spanned by the proposed basis. Then $(\Lambda, Q)$ is an \emph{odd unimodular} lattice of \emph{signature $(1,9)$}. We also have $\Lambda \subseteq \Pic Y \otimes \Z[\frac{1}{3}]$, as an index $3$ sublattice, by construction. Recall that odd unimodular indefinite lattices are \emph{uniquely} determined by their rank. Hence $\Lambda$ coincides with another index $3$ sublattice $\Pic Y \subseteq \Pic Y \otimes \Z[\frac{1}{3}]$. The equality $\Lambda = \Pic Y$ concludes the proof.	
\end{proof}\\

\refstepcounter{equation}
\subsection{One exceptional collection on $Y$}
\label{subsection:pre-3}

Consider the bounded derived category $D^b(Y)$ of coherent sheaves
on $Y$ (for the definitions that follow one may take $Y$ to be any
smooth projective variety). This is a triangulated category, so
that the notions of exceptional object/collection, semiorthogonal
decomposition, etc. make sense (see \cite{bo-ka} and \cite{bo} for
foundations).

Recall that an object $A \in D^b(Y)$ is called \emph{exceptional}
if $\Hom(A,A) = \C$ and $\Ext^i(A,A) = 0$ for all $i \ne 0$. More
generally, a collection $A_1,\ldots,A_N$ of exceptional objects is
called exceptional if $\Ext^{\bullet}(A_i,A_j) = 0$ for all $i >
j$. Given any such collection, one has a \emph{semiorthogonal
decomposition (SOD)}
\[D^b(Y) = \left<\mathcal{A},A_1,\ldots,A_N\right>,\]
where $\mathcal{A} := \left<A_1,\ldots,A_N\right>^{\bot} :=
\left\{A \in D^b(Y)\ \vert \ \Ext^{\bullet}(A_i,A) = 0, 1\le i \le
N\right\}$ is the \emph{right orthogonal complement} to the category $\left<A_1,\ldots,A_N\right>$ generated by $A_i$. The collection is called \emph{full} if $\mathcal{A} =
0$. (N.\,B. In the previous setting, $A_i$ denotes actually the
\emph{category} generated by the corresponding object, $1 \le i
\le N$. The notions of SOD and of being full extend verbatim to
the case of arbitrary strictly full triangulated subcategories
$A_i$.)

\begin{remark}[{cf. \cite[Section 2]{gkms}} or \cite{sosna}]
\label{remark:exc-num-exc} Taking the \emph{Chern character}
$\text{ch}$ of any $A \in D^b(Y)$ delivers a certain numerical
condition (a priori only necessary one) for being exceptional. Namely, the category $D^b(Y)$
becomes replaced by the Abelian group $K_0(Y) :=
K_0(D^b(Y))\ni\text{ch}(A)$ (the \emph{K\,-\,group} of $Y$) and
the complex $\Ext^{\bullet}(\star,\star)$ is replaced by
$\chi(\star,\star) = \displaystyle\sum_k
(-1)^k\dim\Ext^k(\star,\star)$, which is a bilinear form on
$K_0(Y)$.\footnote{Somewhat loosely, we use the symbol $\text{ch}(A)$ to denote both the class of $A$ in $K_0(Y)$ and in the Chow group $\text{CH}^\bullet(Y) \otimes \Q$, baring in mind an isomorphism $K_0(Y) \otimes \Q \simeq \text{CH}^\bullet(Y) \otimes \Q$.} Then the above objects $A_1,\ldots,A_N$ yield a
\emph{numerically exceptional collection}
$\text{ch}(A_1),\ldots,\text{ch}(A_N)$, so that the $N\times
N$\,-\,matrix with entries $\chi(A_i,A_j)$ is
upper\,-\,triangular. Furthermore, if the collection is full, then $K_0(Y)$ is a free group of rank $N =
\dim HH_{\bullet}(D^b(Y)) = c_2(Y)$, $\text{ch}(A_i)$ generate
$K_0(Y)$ and the form $\chi(\star,\star)$ is unimodular. (Note
also that our surface $Y$ admits a numerically exceptional
collection of \emph{maximal length} $c_2(Y) = 12$ --- according to
\cite[Theorem 6]{vial}. The latter is constructed in
\cite[3.1]{vial} by first finding a diagonal basis
$A_2,\ldots,A_{11}$ for the intersection form on $\Pic Y$ (cf.
Lemma~\ref{theorem:3-torsion-classes} above) such that $K_Y =
\displaystyle\sum_{i = 2}^{10}A_i -3A_{11}$, and then taking $A_1
:= \OO_Y$, $A_{12} := 2A_{11}$ to be the rest of the pertinent
collection. Finally, we observe that if the collection
$A_1,\ldots,A_{12}$ is indeed exceptional, then the
HKR\,-\,isomorphism $HH_{k}(Y) \simeq \bigoplus_{q - p = k}
H^{p,q}(Y)$ yields $HH_{\bullet}(\mathcal{A}) = 0$ for higher
Hochschield homology, whereas the equality $K_0(\mathcal{A}) = 0$
follows from the splitting $K_0(Y) = K_0(\mathcal{A}) \oplus \displaystyle\bigoplus_{i = 1}^{12} K_0(A_i)$ and the fact that $A_i$ (freely) generate the K\,-\,group of $Y$.)
\end{remark}

\begin{theorem}
\label{theorem:exc-collection-12} The line bundles
\[
\begin{array}{ccc}
\bigskip
A_1 := \OO_Y, \ A_2 := -\frac{8}{3}\,K_Y - \frac{2}{3}\,R_2 -
\frac{1}{3}\,R_5 - R_8 - \frac{2}{3}\,H_1 + \frac{2}{3}\,H_2 -
\frac{2}{3}\,F_1 + \frac{2}{3}\,F_2,\\
\bigskip
A_3 := \frac{2}{3}\,K_Y - \frac{1}{3}\,R_2 + \frac{1}{3}\,R_5 -
R_8 - H_1 - \frac{2}{3}\,F_1 + \frac{2}{3}\,F_2 - \frac{2}{3}\,G_1
- \frac{1}{3}\,G_2,\\
\bigskip
A_4 := \frac{5}{3}\,K_Y - \frac{1}{3}\,R_2 - \frac{2}{3}\,R_5 -
R_8 - H_1 - \frac{2}{3}\,F_1 + \frac{2}{3}\,F_2 - \frac{2}{3}\,G_1
- \frac{1}{3}\,G_2,\\
\bigskip
A_5 := -R_8 - \frac{1}{3}\,H_1 + \frac{1}{3}\,H_2 -
\frac{2}{3}\,F_1
+ \frac{2}{3}\,F_2 - \frac{1}{3}\,G_1 - \frac{2}{3}\,G_2,\\
\bigskip
A_6 := -\frac{5}{3}\,K_Y - \frac{2}{3}\,R_2 - \frac{1}{3}\,R_5 -
R_8 - \frac{1}{3}\,H_1 + \frac{1}{3}\,H_2 - F_1
+ F_2 - \frac{2}{3}\,G_1 - \frac{1}{3}\,G_2,\\
\bigskip
A_7 := -\frac{5}{3}\,K_Y - \frac{2}{3}\,R_2 - \frac{1}{3}\,R_5 -
R_8 - \frac{1}{3}\,H_1 + \frac{1}{3}\,H_2 + F_2 - \frac{2}{3}\,G_1
- \frac{1}{3}\,G_2,\\
\bigskip
A_8 := 2K_Y + \frac{1}{3}\,H_1 - \frac{1}{3}\,H_2 - \frac{1}{3}\,F_1 - \frac{2}{3}\,F_2 + \frac{1}{3}\,G_1 + \frac{2}{3}\,G_2,\\
\bigskip
A_9 := -2K_Y - R_8 - \frac{1}{3}\,H_1 + \frac{1}{3}\,H_2 -
\frac{2}{3}\,F_1 + \frac{2}{3}\,F_2 - \frac{1}{3}\,G_1 +
\frac{1}{3}\,G_2,\\
\bigskip
A_{10} := \frac{1}{3}\,K_Y - \frac{2}{3}\,R_2 - \frac{1}{3}\,R_5 -
R_8 -
\frac{2}{3}\,H_1 - \frac{1}{3}\,H_2 - \frac{2}{3}\,F_1 + \frac{2}{3}\,F_2, \\
\bigskip
A_{11} := 6K_Y + R_8 + H_1 - 2F_2,\\
\bigskip
A_{12} := \frac{7}{3}\,K_Y - \frac{2}{3}\,R_2 - \frac{1}{3}\,R_5 -
R_8 - \frac{1}{3}\,H_1 + \frac{1}{3}\,H_2 - F_1 - \frac{2}{3}\,G_1
- \frac{1}{3}\,G_2
\end{array}\]

form an exceptional collection in $D^b(Y)$ of length
$12$.\footnote{~The fact that this collection is
\emph{numerically} exceptional follows by direct computation with
PARI/GP. Note that in our special situation this
construction of maximal length numerically exceptional
collection is different from the one described in
Remark~\ref{remark:exc-num-exc}.}
\end{theorem}

\begin{proof}
Recall that $\Ext^{\bullet}(A_i,A_j) := H^{\bullet}(Y,A_i^*\otimes
A_j)$. Here $A_i^*$ denotes the dual, or negative, of $A_i$ (we
will not distinguish between divisors and the corresponding line
bundles). On proving the claimed vanishings for all $i
> j$ we will proceed in a case\,-\,by\,-\,case manner starting
with $j = 1$.

One easily checks via Proposition~\ref{theorem:curves-on-y} that
\begin{equation}
\nonumber \chi(A_i,A_1) := \chi(Y,A_i^*) = 0
\end{equation}
for all $1 < i \le 12$. Hence in order for
$\Ext^{\bullet}(A_i,A_1)$ to vanish it suffices to show that
$H^0(Y,A_i^*) = H^2(Y,A_i^*) = 0$.

\begin{lemma}
\label{theorem:aux-a-1} $H^2(Y,A_i^*) = 0$ for all $i \ne 11$ and
$H^0(Y,A_8^*) = H^0(Y,A_{11}^*) = 0$.
\end{lemma}

\begin{proof}
Observe that $H^2(Y,A_i^*) = H^0(Y,K_Y + A_i)$ by Serre duality.
Now, if $H^2(Y,A_i^*)\ne 0$ for some $i \ne 8,11$, then there
exists a curve $Z \sim K_Y + A_i$ on $Y$. At the same time, we
have $Z \cdot K_Y = -1$ by definition of $A_i$ (cf.
Proposition~\ref{theorem:curves-on-y}), which contradicts $K_Y$ being nef (see \eqref{can-bun}). Same
argument shows that $H^0(Y,A_{11}^*) = 0$.

Further, if $i = 8$ and $H^2(Y,A_8^*)\ne 0$, then there exists a
curve \[Z \sim K_Y + A_8 = 3K_Y + \frac{1}{3}\,H_1 -
\frac{1}{3}\,H_2 - \frac{1}{3}\,F_1 - \frac{2}{3}\,F_2 +
\frac{1}{3}\,G_1 + \frac{2}{3}\,G_2.\] Note that $Z \cdot H_1 = Z
\cdot G_2 = -1$ and hence $Z_1 := Z - H_1 - G_2$ is another curve.
Moreover, we have $Z_1 \cdot K_Y = 0$, which implies that $Z_1$
belongs to fibers of the morphism $p_{\scriptscriptstyle Y}: Y
\lra \P^1$ (cf. {\ref{subsection:pre-2}}). On the other hand, from
\[
Z_1 \cdot R_8 = 3K_Y \cdot R_8 - H_1 \cdot R_8 - G_2 \cdot R_8 =
-1\] we get $R_8 \subseteq Z_1$, a contradiction.

Finally, for $A_8^* \cdot R_8 = -2$ and $A_8^* \cdot K_Y = 0$ one
obtains $H^0(Y,A_8^*) = 0$ exactly as above, which concludes the
proof.
\end{proof}\\

Fix some $i \ne 11$ and suppose that $H^0(Y,A_i^*) \ne 0$ (cf.
Lemma~\ref{theorem:aux-a-1}). Let $Z \sim A_i^*$ be a curve. We
compute $Z \cdot K_Y = 1$ and $(Z^2) = -1$ by construction. Then
there exists a \emph{unique} irreducible component $Z_h \subseteq
Z$ not contracted by the morphism $p_{\scriptscriptstyle Y}$.

Let us establish a couple of supplementary results:

\begin{lemma}
\label{theorem:aux-a-3} In the previous setting, if $Z_h \ne Z$,
then the curve $Z_v := Z - Z_h$ belongs to \emph{reducible} fibers
of $p_{\scriptscriptstyle Y}$ and satisfies $(Z_v^2) < 0$, $Z_v
\cdot Z \le 0$. Furthermore, we have $\deg
\omega_{\scriptscriptstyle Z_h} \ge -2$ for the dualizing sheaf of
$Z_h$, with equality holding if and only if $Z_h \simeq \P^1$.
\end{lemma}

\begin{proof}
The fact that $Z_v$ belongs to the fibers of
$p_{\scriptscriptstyle Y}$ follows from \eqref{can-bun} and $Z_v
\cdot K_Y = 0$ for $Z_h \cdot K_Y = 1$. Suppose next that $Z_v
\cdot Z
> 0$. Then, since the intersection form is even and positive
semi\,-\,definite on the components of the fibers of
$p_{\scriptscriptstyle Y}$, we obtain $(Z_v^2) \le 0$,
\[(Z_h^2) = (Z^2) - 2Z \cdot Z_v + (Z_v^2) < -3\]
and \[\deg \omega_{\scriptscriptstyle Z_h} = (K_Y + Z_h) \cdot Z_h
< -2.\] On the other hand, we have $\deg
\omega_{\scriptscriptstyle Z_h} = 2h^1(Z_h,\OO_{Z_h}) - 2$ by
Riemann -- Roch and Serre duality (see e.\,g. \cite[II.3,
II.6]{bpv}), so that $\deg \omega_{\scriptscriptstyle Z_h} \ge
-2$, a contradiction (this also proves the last assertion of
lemma).

Finally, if $(Z_v^2) = 0$, then $Z_v$ consists of (scheme) fibers
of $p_{\scriptscriptstyle Y}$ (cf. \cite[III.8, Lemma 8.2
(10)]{bpv}), and we get contradiction with $Z_v \cdot Z \le 0$ for
$Z_h$ being a horizontal component.
\end{proof}\\

\begin{proposition}
\label{theorem:not-e-or-p-1} Normalization of the curve $Z_h$ can
not be rational or elliptic.
\end{proposition}

\begin{proof}
Let $\mu_3$ be the Galois group of the covering $\psi: Y \lra W$
from {\ref{subsection:pre-1}}. Then the curves $Z$ and $Z_h$ are
not $\mu_3$\,-\,invariant, since $R_8$ \emph{is}
$\mu_3$\,-\,invariant and there is no non\,-\,trivial linear
relation between the classes of $R_l,H_i,F_h,G_k$. Hence the cycle
$\psi^*\psi(Z_h)$ splits on $Y$.

Suppose first that $Z_h \simeq \P^1$ and $\psi(Z_h)$ is also
smooth. Associate with the morphism $p_{\scriptscriptstyle Y}: Y
\lra \P^1$ a rational $\mu_3$\,-\,invariant function $f \in \C(Y)$
in the usual way. This yields the corresponding (elliptic)
morphism $p_{\scriptscriptstyle W}: W \lra \P^1$ such that
$p_{\scriptscriptstyle Y}\big\vert_{Z_h} = p_{\scriptscriptstyle
W}\big\vert_{\psi(Z_h)}$ via $Z_h \simeq \psi(Z_h)$ induced by
$\psi$. Note that $Z_h \cdot 3F = 6$ by \eqref{can-bun} and the
divisor $(p_{\scriptscriptstyle
W}\big\vert_{\psi(Z_h)})^*(p_{\scriptscriptstyle W}(F))$ is
divisible by $3$ because $\psi^*\psi(Z_h)$ splits. Hence
$p_{\scriptscriptstyle Y}\big\vert_{Z_h}$ corresponds to the
rational function \beq\label{funct-1}f\big\vert_{Z_h} = \frac{(t -
a)^3(t - b)^3}{\displaystyle\prod_{i = 1}^3(t - a_i)^2}\eeq for
some projective coordinate $t$ and $a,b \in F$, $a_i \in C$. Let
$Z'_h \subset \psi^*\psi(Z_h)$ be another irreducible component.
We have $Z'_h \cap Z_h \cap F = \{a,b\}$ because $\psi$ (totally)
ramifies on $F$. In particular, $a,b$ are fixed by $\mu_3$. We
also have \beq\label{funct-2}f\big\vert_{Z'_h} = \frac{(t' -
a)^3(t' - b)^3}{\displaystyle\prod_{i = 1}^3(t' - a_i)^2} =
f\big\vert_{Z_h}\eeq by construction for some projective
coordinate $t'$ and $a,b,a_i$ as above. Furthermore, the M\"obius
transformation $t \mapsto t'$ preserves two points $a,b$ and has
the order dividing $3$, which implies that $t' = \varepsilon t$
for $\varepsilon := \sqrt[3]{1}$. Then from \eqref{funct-1} and
\eqref{funct-2} we obtain that $a = 0$, $b = \infty$ and
$\varepsilon$ permutes $a_1,a_2,a_3$. In particular, the set of
points $a_1,a_2,a_3 \in C$ is $\mu_3$\,-\,invariant, which implies
that $a_i$ form an orbit of length $3$ (cf.
Proposition~\ref{theorem:s-is-jac-y} below). But then we get
$\psi(a_1) = \psi(a_2) = \psi(a_3) =: c$ by definition of $\psi$,
which implies that \[f\big\vert_{\psi(Z_h)} = \frac{(t - a)^3(t -
b)^3}{(t - c)^6},\] contradicting the fact that $f\big\vert_{Z_h}$
is not the $3^{\text{rd}}$ power of a rational function.

Suppose next that $Z_h \simeq \P^1$ and $\psi(Z_h)$ is singular.
Then both $Z_h$ and $Z'_h$ are normalizations of $\psi(Z_h)$.
Again from $Z_h \cdot 3F = 6$ we deduce the expressions
\eqref{funct-1}, \eqref{funct-2}, which gives $f = [3^{\text{rd}}\
\text{power}]$ exactly as above, a contradiction. Finally, if
$Z_h$ is rational and singular, then we have $\deg
\eta^*\OO_{Z_h}(3F) = Z_h \cdot 3F = 6$ for the normalization
$\eta: Z^{\text{norm}}_h \lra Z_h$ (cf. \cite[Example
1.2.3]{fulton}). Hence all the preceding arguments apply verbatim
to the normalization of $Z_h$ and $Z'_h$.

We now turn to the elliptic case. Let us treat only smooth $Z_h$
because the general case is reduced to this one via the
normalization exactly as for rational $Z_h$. Similarly as above,
we obtain \beq\label{funct-3}3a + 3b \sim 2a_1 + 2a_2 + 2a_3\eeq
on $Z_h$ for some $a,b \in F$, $a_i \in C$. Regard $Z_h$ as a
cubic in $\P^2$. Applying an automorphism one may assume $a \in
Z_h$ to be an inflection point. Let $L$ be a linear form defining
the tangent line to $Z_h$ at $a$. In addition, denote by $L_i$ a
linear form defining the tangent at $a_i$, $1 \le i \le 3$. Then
it follows from \eqref{funct-3} that $f\big\vert_{Z_h}$ coincides
with the rational function $LQ\slash(\displaystyle\prod_{i = 1}^3
L_i)$ on $Z_h$, where $(Q = 0) \subset \P^2$ is some conic
intersecting $Z_h$ at $b$ (with multiplicity $3$) and at the
points $Z_h \cap L_i \setminus \{a_i\}$, $1 \le i \le 3$. Now,
arguing exactly as in the case of rational $Z_h$, we get
$$
f\big\vert_{Z'_h} = L'Q'\slash(\displaystyle\prod_{i = 1}^3 L'_i)
= f\big\vert_{Z_h}
$$
for another irreducible component $Z'_h \subset \psi^*\psi(Z_h)$.
Here the forms $L',Q'$ and $L'_i$ are related with $L,Q$ and
$L_i$, respectively, by some projective transformation of $\P^2$
of order $3$, which also fixes $a$ and $b$. In particular, the set
$\{a_1,a_2,a_3\}$ is again a $\mu_3$\,-\,orbit of length $3$, so
that $L_1 = L_2 = L_3$ as previously (recall that
$\psi\big\vert_{Z_h}$ is a birational morphism), i.\,e. $a_1 = a_2
= a_3$, a contradiction.

Proposition~\ref{theorem:not-e-or-p-1} is completely proved.
\end{proof}\\

\begin{corollary}
\label{theorem:not-e-or-p-1-cor} The curve $Z$ is
\emph{reducible}.
\end{corollary}

\begin{proof}
Indeed, otherwise $Z = Z_h$ has $\deg \omega_{\scriptscriptstyle
Z} = 0$ by adjunction, contradicting
Proposition~\ref{theorem:not-e-or-p-1}.
\end{proof}\\

We proceed with our case\,-\,by\,-\,case analysis excluding the
possibilities for the curve $Z \sim A_i^*$.

\begin{lemma}
\label{theorem:exc-collection-12-a-1} Let $Z_h \ne Z$. Then $i \ne
2$.
\end{lemma}

\begin{proof}
Assume the contrary. One checks via
Proposition~\ref{theorem:curves-on-y} the intersections of $Z \sim
A_2^*$ with all fiber components of the morphism
$p_{\scriptscriptstyle Y}$ (cf. {\ref{subsection:pre-2}}) are
\emph{positive}, except for $Z \cdot R_2 = -1$, $Z \cdot H_1 = Z
\cdot F_1 = 0$. Note also that $Z \cdot R_8 = 2$. Let us show that
only $R_2,H_1,F_1$ can belong to $Z_v$.

Indeed, if $Z$ has $H_2$, say, as its irreducible component, so
that $Z - H_2$ is again a curve, then we get
\[(Z - H_2)
\cdot R_8 = 0 = Z_h \cdot R_8 + (Z_v - H_2) \cdot R_8.\] This
implies that either $Z_h = R_8$ and $Z_v - H_2 = R_5 + mR_2$ or
$Z_h \cdot R_8 = 0$ and $Z_v - H_2 = mR_2$ for some (positive) $m
\in \Z$. In both cases, we have \[0 = Z \cdot H_1 \ge Z_v \cdot
H_1 = 1,\] a contradiction. Similarly, replacing $H_2$ by the
curve $\Phi_2 - H_1 - H_2$ (cf. Figure~\ref{fig-3}), we get
contradiction. Exactly the same reasoning also shows that the
curves $F_2$ and $\Phi_3 - F_1 - F_2$ are not among the components
of $Z$.

Further, if $G_1 \subset Z$, then we have $(Z - G_1) \cdot R_8 =
0$, which implies that $Z = R_8 + G_1 + R_5 + mR_2$ or $Z_h + G_1
+mR_2$, similarly as above. In the former case, we obtain
\[-1 = (Z^2) = 1 + 2m - 2m^2,\]
a contradiction. On the other hand, if $Z = Z_h + G_1 +mR_2$, then
\[(Z_h^2) = (Z - G_1 - mR_2)^2 = -3 - 2Z \cdot G_1 + 2m - 2m^2 < -3\]
for $Z \cdot G_1 > 0$, which contradicts
Lemma~\ref{theorem:aux-a-3}. The same applies to the curves $G_2$ and
$\Phi_4 - G_1 - G_2$.

Now, if $R' \subseteq Z_v$ for $R' := \Phi_1 - R_2 - R_5$, then
$-3 = (Z - R') \cdot R_8 \ge -1$ as above, a contradiction.
Finally, if $R_5 \subseteq Z$, then from $(Z - R_5) \cdot R_8 = 1$
and the preceding results we deduce the following possibilities
for $Z$:
\[R_8 + R_5 + H_1 + mR_2, \qquad R_8 + R_5 + F_1 + mR_2, \qquad
Z_h + R_5 + mR_2
\]
($Z_h \cdot R_8 = 1$ in the last case). In the first two cases, we
get \[-1 = (Z^2) = 1 + 2m - 2m^2,\] a contradiction. In turn, if
$Z = Z_h + R_5 + mR_2$, then
\[(Z_h^2) = (Z - R_5 - mR_2)^2 = -3 - 2Z \cdot R_5 + 4m - 2m^2 \le -1\]
for $Z \cdot R_5 \ge 0$, which contradicts
Lemma~\ref{theorem:aux-a-3} and
Proposition~\ref{theorem:not-e-or-p-1}.

Thus we obtain \[Z = Z_h + aR_2 + bH_1 + cF_1
\] (as $1$\,-\,cycles) by construction and Lemma~\ref{theorem:aux-a-3}, where
$a,b,c$ are non\,-\,negative integers and $a > 0$. On the other
hand, we have
\[(Z_h^2) = (Z - aR_2 - bH_1 - cF_1)^2 = -1 + 2a - 2a^2 - 2b^2 - 2c^2 \le -1 + 2a - 2a^2,\]
which gives $(Z_h^2) = -1$ because $K_Y \cdot Z_h = 1$ and $\deg
\omega_{\scriptscriptstyle Z_h} \ge -2$. But then $\deg
\omega_{\scriptscriptstyle Z_h} = 0$ by definition, i.\,e.
normalization of $Z_h$ is either rational or elliptic, which
contradicts Proposition~\ref{theorem:not-e-or-p-1}.

The proof that $H^0(Y,A_2^*) = 0$ under the assumption $Z_h \ne Z$
is complete.
\end{proof}\\

It follows from Lemma~\ref{theorem:exc-collection-12-a-1} that for
$i = 2$ the curve $Z$ must be irreducible --- in contradiction
with Corollary~\ref{theorem:not-e-or-p-1-cor}. Together with
Lemma~\ref{theorem:aux-a-1} we have thus proven the vanishing of
$\Ext^{\bullet}(A_2,A_1)$.

Suppose next that $i = 3$. Then again for $Z \sim A_3^*$ one gets:
$Z \cdot R_2 = -1$, $Z \cdot R_8 = 2$, $Z \cdot H_1 = Z \cdot F_1
= 0$ and the intersections of $Z$ with all other fiber components
of $p_{\scriptscriptstyle Y}$ are positive. Now the arguments from
the case of $i = 2$ apply literally here to show that
$\Ext^{\bullet}(A_3,A_1) = 0$.

\begin{lemma}
\label{theorem:exc-collection-12-a-2} Let $Z_h \ne Z$. Then $i \ne
4$.
\end{lemma}

\begin{proof}
Assume the contrary (so that $Z \sim A_4^*$). One checks via
Proposition~\ref{theorem:curves-on-y} that $Z \cdot R_2 = Z \cdot
R_5 = Z \cdot H_1 = Z \cdot F_1 = 0$ and the intersections of $Z$
with all other fiber components of $p_{\scriptscriptstyle Y}$ are
positive. Note also that $Z \cdot R_8 = 2$. Then all
``\,positive\,'' components of $Z_v$ are excluded exactly as in
the proof of Lemma~\ref{theorem:exc-collection-12-a-1}.

Thus we obtain
\[Z = Z_h + aR_2 + bR_5 + cH_1 + dF_1\]
by construction and Lemma~\ref{theorem:aux-a-3}, where $a,b,c,d
\in \Z$ are non\,-\,negative, $a + b + c + d \ne 0$. On the other
hand, we have
\[(Z_h^2) = (Z - aR_2 - bR_5 - cH_1 - dF_1)^2 = -1 + 2ab - 2a^2 - 2b^2 - 2c^2 - 2d^2,\]
which gives $(Z_h^2) = -3$ and $Z_h \simeq \P^1$, contradicting
Proposition~\ref{theorem:not-e-or-p-1}.

Hence $H^0(Y,A_4^*) = 0$ when $Z_h \ne Z$.
\end{proof}\\

It follows from Lemma~\ref{theorem:exc-collection-12-a-2} that for
$i = 4$ the curve $Z$ must be irreducible --- in contradiction
with Corollary~\ref{theorem:not-e-or-p-1-cor}. Together with
Lemma~\ref{theorem:aux-a-1} this proves the vanishing of
$\Ext^{\bullet}(A_4,A_1)$.

\begin{lemma}
\label{theorem:more-exc-collection-12-a-1} Let $Z_h \ne Z$. Then
$i \ne 5$.
\end{lemma}

\begin{proof}
Assume the contrary (so that $Z \sim A_5^*$). One checks via
Proposition~\ref{theorem:curves-on-y} that $Z \cdot R_2 = Z \cdot
F_1 = 0$ and the intersections of $Z$ with all other fiber
components of $p_{\scriptscriptstyle Y}$ are positive. Note also
that $Z \cdot R_8 = 1$. Then all ``\,positive\,'' components of
$Z_v$ are excluded similarly as in the proof of
Lemma~\ref{theorem:exc-collection-12-a-1} (cf. the case of $i = 6$
below).

Thus we obtain
\[Z = Z_h + aR_2 + bF_1\]
by construction and Lemma~\ref{theorem:aux-a-3}, where $a,b \in
\Z$ are non\,-\,negative, $a + b \ne 0$. On the other hand, we
have
\[(Z_h^2) = (Z - aR_2 - bF_1)^2 = -1 - 2a^2 - 2b^2,\]
which gives $(Z_h^2) = -3$ and $Z_h \simeq \P^1$, contradicting
Proposition~\ref{theorem:not-e-or-p-1}.

Hence $H^0(Y,A_5^*) = 0$ when $Z_h \ne Z$.
\end{proof}\\

\begin{lemma}
\label{theorem:more-exc-collection-12-a-2} Let $Z_h \ne Z$. Then
$i \ne 6$.
\end{lemma}

\begin{proof}
Assume the contrary (so that $Z \sim A_6^*$). One checks via
Proposition~\ref{theorem:curves-on-y} that $Z \cdot R_2 = Z \cdot
F_1 = -1$ and the intersections of $Z$ with all other fiber
components of $p_{\scriptscriptstyle Y}$ are positive. Note also
that $Z \cdot R_8 = 3$. Then the curve $Z_1 := Z - F_1$ satisfies
$Z_1 \cdot R_8 = 1$. Furthermore, $Z_1$ can not contain the curve
$\Phi_1 - R_2 - R_5$, as well as the curves from the fibers
$\Phi_i$, $i > 1$ . Indeed, if $H_1 \subset Z_1$, say, then from
$(Z_1 - H_1) \cdot R_8 = -1$ we deduce that $Z_1 = R_8 + H_1 +
mR_2$ for some $m \in \Z$, which gives $Z_1 \cdot R_2 = -2m^2$
--- in contradiction with $Z \cdot R_2 = -1$.

Thus one gets only two possibilities for $Z_1$ (and $Z$): $Z_h +
mR_2$ or $Z_h + R_5 + mR_2$. In particular, we obtain $(Z_h^2) \le
-1 + 2m - 2m^2 \le -1$ for $Z \cdot R_5 \ge 0$, contradicting
Lemma~\ref{theorem:aux-a-3} and
Proposition~\ref{theorem:not-e-or-p-1}.

Hence $H^0(Y,A_6^*) = 0$ when $Z_h \ne Z$.
\end{proof}\\

Lemmas~\ref{theorem:more-exc-collection-12-a-1} and
\ref{theorem:more-exc-collection-12-a-2} lead to contradiction
with Corollary~\ref{theorem:not-e-or-p-1-cor} exactly as above.

Further, if $i > 6$ is not equal $11$ (cf.
Lemma~\ref{theorem:aux-a-1}), then we have $Z \cdot R_8 = 1$, $Z
\cdot R_2 \ge -1$, $Z \cdot F_1 \ge 0$ and the intersections of
$Z$ with all other fiber components of $p_{\scriptscriptstyle Y}$
are positive. Now the arguments, similar to those in the proofs of
Lemmas~\ref{theorem:more-exc-collection-12-a-1} and
\ref{theorem:more-exc-collection-12-a-2}, lead to contradiction.
We have thus proven that $\Ext^{\bullet}(A_i,A_1) = 0$ for all $i
> 1$ not equal to $11$.

Let us finally show that $\Ext^{\bullet}(A_{11},A_1) = 0$ as well.
According to Lemma~\ref{theorem:aux-a-1} it suffices to prove that
$H^2(Y,A_{11}^*) = 0$. Assume on the contrary that $K_Y + A_{11}
\sim \ [\text{some curve} \ Z]$. One checks via
Proposition~\ref{theorem:curves-on-y} that $Z \cdot R_2 = Z \cdot
H_1 = Z \cdot F_1 = Z \cdot F' = 0$ for $F' := \Phi_3 - F_1 - F_2$
and the intersections of $Z$ with all other fiber components of
$p_{\scriptscriptstyle Y}$ are positive. Note also that $Z \cdot
R_8 = 4$. Let us represent again $Z$ as the sum of horizontal
$Z_h$ and vertical $Z_v$ curves with respect to the morphism
$p_{\scriptscriptstyle Y}$

\begin{lemma}
\label{theorem:more-exc-collection-12-a-3} The curve $Z \sim K_Y +
A_{11}$ is irreducible.
\end{lemma}

\begin{proof}
Assume the contrary. Let us show that only $R_2, H_1, F_1$ and
$F'$ can belong to $Z_v$ (cf. the proof of
Lemma~\ref{theorem:exc-collection-12-a-1}).

Indeed, if $Z$ has $H_2$, say, as its irreducible component, so
that $Z - H_2$ is again a curve, then we get
\[(Z - H_2) \cdot R_8 = 2 = Z_h \cdot R_8 + (Z_v - H_2) \cdot R_8.\]
This yields the following options for $Z_v - H_2$: (i) $R_5 + \Phi
+ mR_2$, $3R_5 + mR_2$ (when $Z_h = R_8$) or (ii) $R_5 + mR_2$,
$2R_5 + mR_2$, $\Phi + mR_2$, $mR_2$ (when $Z_h \cdot R_8 \ge 0$)
for some (positive) $m \in \Z$ and an irreducible curve $\Phi
\subset \Phi_i$, $i
> 1$. It follows from $Z \cdot H_1 = 0$ and $R_8 \cdot H_i = 2$ that only case (ii) is
possible. This implies that
\[(Z_h^2) = (Z - Z_v)^2 = -1 - 2Z \cdot Z_v + (Z_v)^2 < -3,\]
since $Z \cdot Z_v > 0$ by construction, which contradicts
Lemma~\ref{theorem:aux-a-3}.

Further, if $R' = \Phi_1 - R_2 - R_5$ is an irreducible component
of $Z$, then from $(Z - R') \cdot R_8 = -1$ we deduce that $Z =
R_8 + R' + mR_2$, which contradicts $Z \cdot R_2 = 0$. Finally, if
$R_5 \subset Z$, then from $(Z - R_5) \cdot R_8 = 3$ and the
preceding results we deduce the following possibilities for $Z$:
\[R_8 + R_5 + 2H_1 + mR_2,\qquad R_8 + R_5 + 2F_1 + mR_2, \qquad R_8 + R_5 + 2F' +
mR_2,
\]
\[R_8 + R_5 + H_1 + F_1 + mR_2,\qquad R_8 + R_5 + H_1 + F' + mR_2,\qquad R_8 + R_5 + F_1 + F' + mR_2,\]
\[R_8 + 3R_5 + H_1 + mR_2,\qquad R_8 + 3R_5 + F_1 + mR_2,\qquad R_8 + 3R_5 + F' + mR_2, \qquad R_8 + 5R_5 +
mR_2\] (when $Z_h = R_8$) or \[Z_h + R_5 + mR_2,\qquad Z_h + 2R_5
+ mR_2,\qquad Z_h + 3R_5 + mR_2,\qquad Z_h + 4R_5 + mR_2\]
\[Z_h + R_5 + H_1 + mR_2,\qquad Z_h + R_5 + F_1 + mR_2,\qquad Z_h + R_5 + F' + mR_2,\]
\[Z_h + 2R_5 + H_1 + mR_2,\qquad Z_h + 2R_5 + F_1 +
mR_2,\qquad Z_h + 2R_5 + F' + mR_2\] (when $Z_h \cdot R_8 \ge 0$).
In all cases, we either get contradiction with $Z \cdot R_2 = Z
\cdot H_1 = Z \cdot F_1 = Z \cdot F' = 0$ or obtain $(Z_h^2) <
-3$, contradicting Lemma~\ref{theorem:aux-a-3}.

Hence we have \[Z = Z_h + aR_2 + bH_1 + cF_1 + dF',
\]
where $a,b,c,d \in \Z$ are non\,-\,negative, $a + b + c + d \ne
0$. Now exactly as in the proof of
Lemma~\ref{theorem:exc-collection-12-a-2} one gets contradiction
with Proposition~\ref{theorem:not-e-or-p-1}.

The proof that $H^2(Y,A_{11}^*) = 0$ under the assumption $Z_h \ne
Z$ is complete.
\end{proof}\\

It follows from Lemma~\ref{theorem:more-exc-collection-12-a-3}
that $Z = Z_h$. But then from $(Z^2) = -1$ and $\deg
\omega_{\scriptscriptstyle Z} = 0$ we get contradiction exactly as
in the proof of Proposition~\ref{theorem:not-e-or-p-1}. This
concludes the proof that $\Ext^{\bullet}(A_{11},A_1) = 0$.

We now turn to the case of $j > 1$.

\begin{lemma}
\label{theorem:exc-collection-12-a-3}
$\Ext^{\bullet}(A_{12},A_{11}) = \Ext^{\bullet}(A_{11},A_j) = 0$
for all $1 < j < 11$ not equal $8$.
\end{lemma}

\begin{proof}
Let us consider the case of $\Ext^{\bullet}(A_{12},A_{11})$ first.
Again, since
\begin{equation}
\nonumber \chi(A_{12},A_{11}) := \chi(Y,A_{12}^* \otimes A_{11}) =
0
\end{equation}
via Proposition~\ref{theorem:curves-on-y}, it suffices to prove
that $H^0(Y,A_{12}^* \otimes A_{11}) = H^2(Y,A_{12}^* \otimes
A_{11}) = 0$. Moreover, since $(K_Y + A_{12} - A_{11}) \cdot K_Y =
-2$, we immediately get $H^2 = 0$. Thus it remains to show that
$H^0 = 0$.

Assume on the contrary that $A_{11} - A_{12} \sim \ [\text{some
curve} \ Z]$. We have $(Z^2) = 0$ and $Z \cdot K_Y = 2$. Let us
represent $Z$ as the sum of horizontal $Z_h$ and vertical $Z_v$
curves with respect to the morphism $p_{\scriptscriptstyle Y}$
(cf. the above case of $j = 1$).

Note that $Z \cdot R_2 = -1$, i.\,e. $R_2 \subseteq Z_v$, and one
computes the curve $Z - R_2$ intersects any possible irreducible
component of $Z_v$ \emph{non\,-\,negatively}. In particular, if
$Z_v \ne R_2$, then we obtain $(Z_h^2) \le -2$ (cf. the proof of
Lemma~\ref{theorem:aux-a-3}). This together with $Z_h \cdot K_Y =
2$ shows that $\deg \omega_{\scriptscriptstyle Z_h} \le 0$. Hence
the arithmetic genus of (the underlying reduced scheme of) $Z_h$
is either $0$ or $1$. Now the result (and the proof) of
Proposition~\ref{theorem:not-e-or-p-1} apply literally here to
provide a contradiction.

Thus we get $Z = Z_h + R_2$. Note that $(Z_h^2) = (Z^2) = 0$ in
this case. Then for any irreducible component $\widetilde{Z}_h
\subsetneq Z_h$ we have $(\widetilde{Z}_h^2) < 0$ and
$\widetilde{Z}_h \cdot K_Y = 1$. In particular, we obtain that
$\deg \omega_{\scriptscriptstyle \widetilde{Z}_h} \le 0$, which
gives contradiction via Proposition~\ref{theorem:not-e-or-p-1} as
above. Hence the curve $Z_h$ is \emph{irreducible}. Then the
linear system $|Z_h|$ either consists of only $Z_h$ or else it is
basepoint\,-\,free. But in the latter case, since $c_2(X) = 12$
and $\chi_{\text{top}}(Z_h) = -2$ (for generic $Z_h$), the
morphism given by $|Z_h|$ will have a  \emph{singular} fiber of
arithmetic genus $2$, which once again contradicts
Proposition~\ref{theorem:not-e-or-p-1}. Thus we get
$H^0(Y,\OO_Y(Z_h)) = H^0(Y,\OO_Y(Z)) = \C$. This yields a
\emph{unique non\,-\,split} rank $2$ vector bundle $\E$ on Y
corresponding to a non\,-\,trivial class in $\text{Ext}^1(\OO_Y,
\OO_Y(Z)) \simeq \C$.

Choose an affine open subset $U \subset \P^1$ and refer to the
surface $p_{\scriptscriptstyle Y}^{-1}(U)$ as a genus $1$ curve
$E$ over the function field $\C(U)$. Then $\OO_Y$ is identified
with the structure sheaf of $E$ and $\OO_Y(Z)$ corresponds to some
$0$\,-\,cycle $\xi$ on $E$ (of length $12$). In turn, the
restriction $\E\big\vert_{p_{\scriptscriptstyle Y}^{-1}(U)}$ can
be regarded as an extension of $\OO_E$ by $\OO_E(\xi)$, which is
thus trivial due to $\text{Ext}^1(\OO_E, \OO_E(\xi)) =
H^1(E,\OO_E(\xi)) = 0$.  On the other hand, the
$\text{Gal}(\overline{\C(U)} \slash \C(U))$\,-\,action on
$E(\overline{\C(U)})$,\footnote{~For any field {\bf k}, we denote
by $\text{Gal}(\bar{{\bf k}} \slash {\bf k})$ the Galois group of
an algebraic closure $\bar{{\bf k}} \supseteq {\bf k}$, whereas
$E({\bf k})$ denotes the set of ${\bf k}$\,-\,points on the curve
$E$.} with varying $U$, generates the action of the global
monodromy of $p_{\scriptscriptstyle Y}$ (cf.
{\ref{subsection:pha-1}} below). In particular, since the bundle
$\E$ is unique, it is invariant under the
monodromy.\footnote{~Indeed, any element $\tau \in \text{Gal}$
acts on the group $\text{Ext}^1(\OO_Y, \OO_Y(Z))$ in a natural
way, so that $\tau^*\E = \E$.} Then various trivializations
$\E\big\vert_{p_{\scriptscriptstyle Y}^{-1}(U)} \simeq \OO_E
\oplus \OO_E(\xi)$ agree for different choices of $U$, hence they
extend over the whole base $\P^1$ of $p_{\scriptscriptstyle Y}$ to
show that $E$ actually \emph{splits}, a contradiction. Thus we get
$H^0(Y,A_{12}^* \otimes A_{11}) = 0$.

Further, turning to the case of $\Ext^{\bullet}(A_{11},A_j)$,
where $1 < j < 11$ is not equal $8$, we have
\[\chi(A_{11},A_j) := \chi(Y,A_{11}^* \otimes A_{j}) = 0\] via
Proposition~\ref{theorem:curves-on-y}. Then $H^0(Y,A_{11}^*
\otimes A_{j}) = 0$ because of $(A_j - A_{11}) \cdot K_Y = -2$.
Now, assuming that $H^2(Y,A_{11}^* \otimes A_{j}) \ne 0$ for some
$j$, let us fix a curve $Z \sim K_Y + A_{11} - A_j$. Note that
again $(Z^2) = 0$ and $Z \cdot K_Y = 2$.

Consider $j \in \{2,3,7,10\}$. One computes $Z \cdot R_2 = -1$ for
these $j$, i.\,e. $R_2 \subset Z$, and the curve $Z - R_2$
intersects any fiber component of $Z$ non\,-\,negatively. Then
running the preceding arguments verbatim we obtain a contradiction.
Similarly, if $j = 6$, then we have $Z \cdot R_2 = Z \cdot F_1 =
-1$ and the curve $Z - R_2 - F_1$ intersects any fiber component
of $Z$ non\,-\,negatively, whence a contradiction as earlier.
Finally, if $j \in \{4,5,9\}$, then already $Z$ intersects all of
its fiber components non\,-\,negatively, so that $Z$ is horizontal
and irreducible as above, which is again impossible.

The proof of lemma is complete.
\end{proof}\\

\begin{lemma}
\label{theorem:exc-collection-12-a-4} $\Ext^{\bullet}(A_i,A_8) =
\Ext^{\bullet}(A_8,A_j) = 0$ for all $i > 8$ and $1 < j < 8$.
\end{lemma}

\begin{proof}
Similarly as above, since $\chi(A_i,A_8) = \chi(A_8,A_j) = 0$ for
all $i > 8$ and $j < 8$, it suffices to prove that
$H^0(Y,A_i^*\otimes A_8) = H^2(Y,A_i^* \otimes A_8) = H^2(Y,A_j
\otimes A_8^*) = 0$. Note that $H^0(Y,A_j \otimes A_8^*) = 0$
because $(A_j - A_8) \cdot K_Y = -1$. We also have $H^2(Y,A_i^*
\otimes A_8) = 0$, $i \ne 11$, and $H^0(Y,A_{11}^*\otimes A_8) =
0$ by the same reason.

Suppose first that $i \ne 11$ and there exists a curve $Z \sim A_8
- A_i$. One computes $(Z^2) = -1$, $Z \cdot K_Y = 1$, $Z \cdot R_8
= 3$, $Z \cdot R_2 \ge -1$ and $Z$ intersects all other fiber
components of the morphism $p_{\scriptscriptstyle Y}$
non\,-\,negatively. Now similarly as in the proofs of
Lemmas~\ref{theorem:more-exc-collection-12-a-2},
\ref{theorem:more-exc-collection-12-a-3}, we first exclude those
irreducible fiber components that intersect $Z$ positively, and
then derive a contradiction with Lemma~\ref{theorem:aux-a-3} and
Proposition~\ref{theorem:not-e-or-p-1}.

In turn, equality $H^2(Y,A_{11}^* \otimes A_8) = 0$ follows as
previously, by observing that $(K_Y + A_{11} - A_8) \cdot R_8 = 2$
and $K_Y + A_{11} - A_8$ intersects all the fiber components of
$p_{\scriptscriptstyle Y}$ non\,-\,negatively.

Now assume that $K_Y + A_8 - A_j \sim [\text{some curve} \ Z]$.
Again we have $(Z^2) = -1$ and $Z \cdot K_Y = 1$. Firstly, if $j
\le 4$, then $Z \cdot R_8 = 5$. At the same time, we have $Z \cdot
H_1 = -1$, which implies that $Z_1 := Z - H_1$ is a curve. Again
$(Z_1^2) = (Z^2) = -1$, $Z_1 \cdot R_8 = 3$, $Z_1 \cdot R_2 = -1$
and $Z_1$ intersects all other fiber components of
$p_{\scriptscriptstyle Y}$ non\,-\,negatively. Similarly as in the
proof of Lemmas~\ref{theorem:more-exc-collection-12-a-2}, we first
exclude those irreducible fiber components that intersect $Z$
positively, and then derive a contradiction with
Lemma~\ref{theorem:aux-a-3} and
Proposition~\ref{theorem:not-e-or-p-1}.

Secondly, if $j = 6$, then $Z \cdot R_8 = 6$. At the same time, we
have $Z \cdot R_2 = Z \cdot F_1 = -1$, which implies that $Z_1 :=
Z - R_2 - F_1 $ is a curve. Note that $(Z_1^2) = (Z^2) = -1$, $Z_1
\cdot R_8 = 4$ and $Z_1$ intersects all other fiber components of
$p_{\scriptscriptstyle Y}$ non\,-\,negatively. Similarly as in the
proof of Lemma~\ref{theorem:more-exc-collection-12-a-3}, we first
exclude those irreducible fiber components that intersect $Z$
positively, and then derive a contradiction with
Lemma~\ref{theorem:aux-a-3} and
Proposition~\ref{theorem:not-e-or-p-1}.

Finally, if $j = 5$ or $7$, then $Z \cdot R_8 = 4$, $Z \cdot R_2
\ge -1$ and $Z$, $Z - R_2$ intersect all other fiber components of
$p_{\scriptscriptstyle Y}$ non\,-\,negatively. Again the preceding
considerations give a contradiction.

The proof of lemma is complete.
\end{proof}\\

For all other $i > j > 1$, not considered previously (cf.
Lemmas~\ref{theorem:exc-collection-12-a-3} and
\ref{theorem:exc-collection-12-a-4}), we have $K_Y \cdot (A_j -
A_i) = 0$ and again $\chi(A_i, A_j) = 0$. In particular, once
$H^0(Y,A_i^*\otimes A_j) \ne 0$ (resp. $H^2(Y,A_i^*\otimes A_j)
\ne 0$), there exists a curve $Z \sim A_j - A_i$ (resp. $Z \sim
K_Y + A_i - A_j$) contained in the fibers of
$p_{\scriptscriptstyle Y}$. In other words, we have $Z = Z_v$ in
the previous notations, which implies that $A_j - A_i$ (resp. $K_Y
+ A_i - A_j$) is an \emph{integral} combination of $K_Y$ and the
classes of $R_l,H_i,F_h,G_k$. This situation is only possible for
the following values of $(i,j)$:
\[(10,2),\qquad (4,3),\qquad (9,5),\qquad (7,6),\qquad (12,6),\qquad (12,7).\]
The corresponding divisors $A_j - A_i$ are
\[-3K_Y + H_2,\qquad -K_Y + R_5,\qquad 2K_Y - G_2,\qquad -F_1,\qquad -4K_Y + F_2,\qquad -4K_Y + F_1 + F_2.\]
It is now immediate that $H^0(Y,A_i^*\otimes A_j) =
H^2(Y,A_i^*\otimes A_j) = 0$. For example, if $(i,j) = (10,2)$ and
$-3K_Y + H_2 \sim \ [\text{some curve} \ Z]$, then $Z - H_2 \sim
-3K_Y$ is also a curve for $Z \cdot H_2 = -2$, a contradiction. In
turn, if $H^0(Y,4K_Y - H_2) \ne 0$, then from
\eqref{can-bun} we deduce that $2F$ contains $H_2$, which
is also impossible. The remaining cases of $(i,j)$ are treated
similarly and employ in addition the fact that $C + F \sim 5K_Y$
does not contain rational curves.

Theorem~\ref{theorem:exc-collection-12} is completely proved.
\end{proof}\\

\bigskip

\section{A phantom}
\label{section:pha}

\refstepcounter{equation}
\subsection{Description of $\jac\,Y$}
\label{subsection:pha-0-1}

We retain the notation of Section~\ref{section:pre}.

Recall that the \emph{Hesse pencil} is a one\,-\,dimensional
linear system of cubics on $\P^2$ of the form \beq\label{hesse}
t_0(x^3 + y^3 + z^3) + t_1xyz = 0.\eeq Here $x,y,z$ are projective
coordinates on $\P^2$ and $[t_0:t_1]\in\P^1$ parameterizes the
members of the pencil (we refer to \cite{art-dolg} for an overview
of the many properties of \eqref{hesse}).

Blowing up the base points one resolves the indeterminacies of the
rational map $\P^2\dashrightarrow\P^1$ given by the Hesse pencil.
This yields a minimal elliptic surface $p_{\scriptscriptstyle S}:
S \lra \P^1$ whose Mordell -- Weil group is generated by $9$
sections of $p_{\scriptscriptstyle S}$ and is isomorphic to
$(\Z\slash 3\Z)^2$.

\begin{proposition}
\label{theorem:s-is-jac-y} We have $S = \jac\,Y$. In particular,
there is a natural inclusion $\mu_3 \subset (\Z\slash 3\Z)^2
\subset \Aut\,S$, where $\mu_3$ is the Galois group of the
covering $\psi$ (see the end of {\ref{subsection:pre-1}})
\end{proposition}

\begin{proof}
Recall that fibration $p_{\scriptscriptstyle S}$ has exactly $4$
fibers $\Phi_i$ of type $I_3$ and no other singular fibers (the
images of $\Phi_i$ on $\P^2$ form the \emph{Hesse configuration}
of $12$ lines passing through $9$ inflection points common to all
smooth curves \eqref{hesse}).

In turn, the surface $Y$ is obtained from $\jac\,Y$ via the
logarithmic transformations along the fibers corresponding to $C$
and $F$, respectively (cf. \eqref{mu-2-act} below). This implies that the only singular fibers
of the elliptic fibration $\jac\,Y \lra \P^1$ are $4$ curves of
type $I_3$ (cf. Figure~\ref{fig-3} and the beginning of
{\ref{subsection:pre-2}}). Now the equality $S = \jac\,Y$ follows
from the classification of stable elliptic fibrations over $\P^1$ with exactly four singular fibers (see \cite[Tableau]{bea-st-ell-fib}). The last assertion of lemma is due to
the fact that $\mu_3$ preserves the fibers of
$p_{\scriptscriptstyle Y}$ (by construction) and acts naturally on
$\jac\,Y$ (cf. \cite[Section 4]{art-dolg}).
\end{proof}\\

\begin{notation}
All elliptic surfaces we consider in the text may be identified in
a natural way with genus $1$ curves over the field $\C(\P^1)$.
Then for the notation's simplicity we will not distinguish in what
follows between any such surface (over $\C$) and the corresponding
curve (over $\C(\P^1)$). Similarly, all non\,-\,fiber curves
(a.\,k.\,a. \emph{multi\,-\,sections}) on a given elliptic surface
are identified with $\overline{\C(\P^1)}$\,-\,points on the
corresponding elliptic curve, and vice versa.
\end{notation}

\begin{remark}
\label{remark:jac-isog-6} Let $R \subset Y$ be a multi\,-\,section
of some degree $n$ (i.\,e. the induced morphism
$p_{\scriptscriptstyle Y}\big\vert_R$ has degree $n$). Then the
genus $1$ curve $Y\slash{\bf k}$ acquires a point over a degree
$n$ extension ${\bf k} \supset \C(\P^1) =: {\bf k}_0$. This yields
a degree $n$ isogeny $Y \lra \jac\,Y$ of ${\bf k}_0$\,-\,curves.
Indeed, by definition we have an isomorphism $\theta: \jac^n\,Y
\stackrel{\sim}{\to} \jac\,Y$, as well as the Poincar\'e line
bundle $\Theta$ on $Y \times \jac^n\,Y$ (all over ${\bf k}_0$),
where $\jac^n\,Y$ is the moduli space of degree $n$ line bundles
on $Y$. Now the morphism $Y \lra \jac\,Y$, given by $y \mapsto
y^*\Theta\big\vert_{Y \times 0}$ (shifts by all $y \in \bar{{\bf
k}}_0$) and the subsequent composition with $\theta$, is
$\text{Gal}(\bar{{\bf k}}_0\slash{\bf k}_0)$\,-\,equivariant and
maps $R \in Y({\bf k})$ to $0 \in \jac\,Y$. This is exactly the
claimed ${\bf k}_0$\,-\,isogeny $Y \lra \jac\,Y$.\footnote{~More
precisely, as follows from the construction the graph of $Y \lra
\jac\,Y$ is invariant under the action of $\text{Gal}({\bf
k}\slash{\bf k}_0)$, hence also under the action of
$\text{Gal}(\bar{{\bf k}}_0\slash{\bf k}_0)$, which implies that
it is defined over ${\bf k}_0$.} In particular, letting $R := R_8$
(see Proposition~\ref{theorem:curves-on-y}) we get $n = 6$, and it
follows from Proposition~\ref{theorem:s-is-jac-y} that the group
$\mu_3\subseteq\Aut\,Y$ acts on $Y$ via fiberwise shifts by some
$c\in\jac^6\,Y$.
\end{remark}

\refstepcounter{equation}
\subsection{The monodromy of $Y$}
\label{subsection:pha-1}

Put $B := p_{\scriptscriptstyle
Y}(Y\setminus\left\{C,F,\Phi_1,\ldots,\Phi_4\right\})$ (the notations
are as earlier). Then all fibers of the elliptic fibration
$p_{\scriptscriptstyle Y}: Y \lra \P^1$ restricted to $B$ are
smooth. This yields a homomorphism $\pi_1(B)\to\slg_2(\Z)$,
defined up to the conjugation, and its image $\Gamma_Y$ is called
the \emph{(global) monodromy group} (of $p_{\scriptscriptstyle Y}$
or of $Y$). Note that $\Gamma_Y$ is generated by $5$ elements
$g_F,g_1,\ldots,g_4$ (for $\pi_1(\P^1\setminus
\{p_{\scriptscriptstyle Y}(C)\}) = \{1\}$), which satisfy the
relation
\[g_F\cdot g_1\cdot\ldots\cdot g_4 = 1,\] where $g_F \in \slg_2(\Z)$ is the image
of a loop on $B$ around $p_{\scriptscriptstyle Y}(F)$, etc. (see
e.\,g. \cite[Lemma 2.1]{bo-yu}).

Let $\mu_3$ be as in Proposition~\ref{theorem:s-is-jac-y}. Choose
some generator $\sigma\in\mu_3$. It follows from the construction
of $Y$ (via a ``\,cubic root of $F$\,'') that $\sigma\in\Gamma_Y$.
Recall also that the monodromy of $S = \jac\,Y$ is the
\emph{modular group}
\[\Gamma(3) :=
\text{Ker}\big[\slg_2(\Z)\twoheadrightarrow\slg_2(\F_3)\big].\] By
construction there is a natural (outer) action of $\sigma$ on
$\Gamma(3)$. This allows one to identify $\Gamma_Y$ with an
extension of $\mu_3 \ni \sigma$ by $\Gamma(3)$ (cf. 
{\ref{subsection:log-trans-1}} below). We will see in the next
subsection that $\sigma$ ``\,detects\,'' phantoms in $D^b(Y)$.

\refstepcounter{equation}
\subsection{A phantom in $D^b(Y)$}
\label{subsection:pha-2}

Consider the length $12$ exceptional collection ${\bf EC} :=
\left\{A_1,\ldots,A_{12}\right\}$ from
Theorem~\ref{theorem:exc-collection-12}. We will assume that it is
full and bring this to contradiction.

Let $F_{\zeta}$ be the general fiber of $p_{\scriptscriptstyle Y}$
treated as a genus $1$ curve over the field $\C(\P^1)$. Recall that according to the preceding {\it Notation} one may regard all the $A_i$ as sheaves on $F_{\zeta}$. Then we claim that the smallest triangulated subcategory of $D^b(F_{\zeta})$ closed under the direct summands and containing $A_1,\ldots,A_{12}$ is $D^b(F_{\zeta})$ itself (cf.
{\ref{subsection:pre-3}}). This reads as follows:

\begin{lemma}
\label{theorem:ec-d-f-zeta} The collection ${\bf EC}$ generates the category
$D^b(F_{\zeta})$. In other words, the sheaves $\OO_{F_{\zeta}}$,
$\OO_{F_{\zeta}}(R_8)$ and $\OO_{F_{\zeta}}(-R_8)$ generate $D^b(F_{\zeta})$.
\end{lemma}

\begin{proof}
Given any element $A\in\Pic F_{\zeta}$, we regard it as a formal
difference of two multi\,-\,sections of $p_{\scriptscriptstyle
Y}$, which yields an element $A_Y \in \Pic Y$ naturally associated
with $A$. In this way, to every complex of direct sums of line
bundles on $F_{\zeta}$ one associates a complex of direct sums of
line bundles on $Y$, since $\Hom\,(A_1,A_2) :=	
H^0(F_{\zeta},A_1^*\otimes A_2)$ for any $A_i \in \Pic F_{\zeta}$.

Further, any vector bundle on $F_{\zeta}$ is quasi\,-\,isomorphic, as an object in $D^b(F_{\zeta})$, to a complex of direct sums of line bundles (cf. \cite{atiyah}). Thus any object in $D^b(F_{\zeta})$, which is a complex of vector bundles, is naturally an object in $D^b(Y)$. Now applying the assumption $D^b(Y) = \left<A_1,\ldots,A_{12}\right>$ and restricting back to $F_{\zeta}$ we obtain the first statement of lemma. The last statement follows from the fact that $A_i\big\vert_{F_{\zeta}}$ is either $\OO_{F_{\zeta}}$ or
$\OO_{F_{\zeta}}(\pm R_8)$ for any $1 \le i \le 12$.
\end{proof}\\

It follows from
Lemma~\ref{theorem:ec-d-f-zeta} that the $K$\,-\,group $K_0(F_{\zeta})$ is generated by $\OO_{F_{\zeta}}$ and $\OO_{F_{\zeta}}(R_8)$ (cf. Remark~\ref{remark:exc-num-exc}).
Further, for the automorphism $\sigma$ from {\ref{subsection:pha-1}}
we have $\sigma^*\OO_{F_{\zeta}} = \OO_{F_{\zeta}}$, $\sigma^*\OO_{F_{\zeta}}(R_8) = \OO_{F_{\zeta}}(R_8)$ by construction. Thus the induced action of $\sigma$ on $K_0(F_{\zeta})$ is \emph{trivial}. At the same time, for the effective $0$\,-\,cycle $c$ on $F_{\zeta}$ from
Remark~\ref{remark:jac-isog-6} we have $\sigma(c) = c + c \ne c$ (for $\sigma$ being an order $3$ shift by $c$), i.\,e. $\sigma$ acts \emph{non\,-\,identically} on $\OO_{F_{\zeta}}(c) \in K_0(F_{\zeta})$, a contradiction. 

Hence the collection ${\bf EC}$ is not full in $D^b(Y)$ and there exists a
phantom $\mathcal{A} = \left<A_1,\ldots,A_{12}\right>^{\bot} \ne
0$ (cf. the beginning of {\ref{subsection:pre-3}}).\,\footnote{~It
would be interesting to distinguish the (derived) class of the
curve $F_{\zeta}$ (resp. describe $\mathcal{A}$ in terms of
$F_{\zeta}$) using the results and technique of the paper
\cite{akw}.}

\begin{remark}
\label{remark:monodromy-gen-phantom} Since $q(Y) = p_g(Y) = 0$ and
the class $K_Y$ is deformation\,-\,invariant, one can easily prove
that any (sufficiently general) algebraic surface $Y'$ in the same
deformation class as $Y$ is again a $(2,3)$\,-\,Dolgachev,
carrying an exceptional collection $A'_1,\ldots,A'_{12}$ such that
$A'_i$ deforms to $A_i$ for all $i$. Furthermore, both fibrations
$p_{\scriptscriptstyle Y'}: Y' \lra \P^1$ and
$p_{\scriptscriptstyle Y}: Y \lra \P^1$ are relatively
diffeomorphic, with any such diffeomorphism $s: Y\slash\P^1
\stackrel{\sim}{\to} Y'\slash\P^1$ satisfying $s(A_i) = A'_i,1\le
i\le 12$, and $s(\Gamma_Y) = \Gamma_{Y'}$ (cf.
{\ref{subsection:pha-1}}). This gives a phantom $\mathcal{A'}
\subset D^b(Y')$ exactly as 
above.\footnote{~We will develop more general (and conceptual)
arguments to construct exceptional collections and phantoms in
Section~\ref{section:log-trans} below.}
\end{remark}

\bigskip

\section{Log\,-\,transfer}
\label{section:log-trans}

\refstepcounter{equation}
\subsection{Setup}
\label{subsection:log-trans-1}

We are going to extend the result of
Theorem~\ref{theorem:exc-collection-12} (cf.
Remark~\ref{remark:monodromy-gen-phantom}) to the case of an
arbitrary $(2,3)$\,-\,Dolgachev surface $\frak{S}$ with
$\jac\,\frak{S} = S$ (notation is as in
{\ref{subsection:pha-0-1}}). For technical reasons, though, we had
to impose an {\it Assumption} at the end of this subsection.

Fix two distinct points $p$ and $q$ on $\mathbb{P}^1$ such that
the fibers $F_p := p_{\scriptscriptstyle S}^{-1}(p)$ and $F_q :=
p_{\scriptscriptstyle S}^{-1}(q)$ are smooth. Recall that
$\frak{S}$ is obtained by \emph{logarithmic transformations} along
$F_p$ and $F_q$ with multiplicities $2$ and $3$, respectively, as
follows (see \cite[V.13]{bpv} for details). Take $p$ for instance
and a small disk $\Delta_p \subset \mathbb{P}^1$ around it so that
all fibers of $p_{\scriptscriptstyle S}$ are smooth over
$\Delta_p$. Choose a trivialization of the family of lattices
$H^1(p_{\scriptscriptstyle S}^{-1}(t),\mathbb{Z}) \subset
\mathbb{C}$, where $t \in \Delta_p$ is a parameter with $p = \{t =
0\}$, and identify
$$
p_{\scriptscriptstyle S}^{-1}(\Delta_p) = \{(z,t)\ \vert \ z \in
\mathbb{C} \slash H^1(p_{\scriptscriptstyle
S}^{-1}(t),\mathbb{Z}), \ t \in \Delta_p\}.
$$
Then there is a $\mu_2$\,-\,action on $p_{\scriptscriptstyle
S}^{-1}(\Delta_p)$ as follows: \beq\label{mu-2-act} (z,t) \mapsto
(z + \frac{1}{2},-t).\eeq The quotient $p_{\scriptscriptstyle
S}^{-1}(\Delta_p)\slash\mu_2$ is a smooth elliptic surface over
$\Delta_p$ with unique singular fiber over $p = 0$ that is
multiple of multiplicity $2$. The surfaces $p_{\scriptscriptstyle
S}^{-1}(\Delta_p)\slash\mu_2$ and $S \setminus F_p$ can be glued
into a smooth compact elliptic surface. Applying further the same
construction with $q$ and $\mu_3$ one obtains $\frak{S}$.

Denote by $p_{\scriptscriptstyle \frak{S}}: \frak{S} \lra \P^1$
for what follows the corresponding elliptic fibration. We will
also denote the fibers of $p_{\scriptscriptstyle \frak{S}}$ over
$p$ and $q$ by the same symbols as above. The surface $\frak{S}$
thus admits an open cover by analytic subsets $U_p :=
p_{\scriptscriptstyle \frak{S}}^{-1}(\Delta_p)$, $U_q :=
p_{\scriptscriptstyle \frak{S}}^{-1}(\Delta_q)$ and $U_{pq} :=
\frak{S}\setminus\{F_p \cup F_q\}$, with $\frak{S}\setminus\{U_p
\cup U_q\} \simeq S\setminus\{p_{\scriptscriptstyle
S}^{-1}(\Delta_p) \cup p_{\scriptscriptstyle S}^{-1}(\Delta_q)\}$
as \emph{complex manifolds}. The principal goal of this section is
to understand when a given triple
\begin{equation}
\label{triple} (L_p \in \Pic U_p, L_q \in \Pic U_q, L_{pq} \in
\Pic U_{pq})
\end{equation}
of line bundles is induced by restriction from a line bundle on
$\frak{S}$ (this will ultimately allow one relate exceptional
collections in $D^b(S)$ and $D^b(\frak{S})$). The key role here
will be played by the monodromy group $\Gamma_{\frak{S}}$ of
$p_{\scriptscriptstyle \frak{S}}: \frak{S} \lra \P^1$ for which we
make the following additional

\begin{assumption}
$\Gamma_{\frak{S}}$ fits into an exact sequence \beq\nonumber 1
\to \Gamma(3) \to \Gamma_{\frak{S}} \to \mu_3 \to 1. \eeq Here
$\mu_3$ gives the monodromy action around $F_q$ and $\Gamma(3)$ is
induced by the monodromy of $p_{\scriptscriptstyle S}$.
\end{assumption}

\begin{remark}
\label{remark:comment-on-assum} The {\it Assumption} is crucial
for the coming constructions (see e.\,g. the end of
{\ref{subsection:log-trans-2}}). Note that it is met for one
particular, hence generic, surface $\frak{S}$ (see
{\ref{subsection:pha-1}} and
Remark~\ref{remark:monodromy-gen-phantom}). It is not at all clear
however whether the extension structure of $\Gamma_{\frak{S}}$
takes place for \emph{any} given $\frak{S}$ (compare with
\cite[Section 3]{bo-pet-yu}). We will give an outline towards the
evidence for such a structure later in the Appendix
below.\footnote{~This Appendix aims to put the present technical
results in a broader perspective (yet the complete arguments will
appear elsewhere).}
\end{remark}

Recall that $\frak{S}$ is simply\,-\,connected. Furthermore, it is
glued out of $U_p$, $U_q$ and $U_{pq}$, so that for the triple
\eqref{triple} to come from a line bundle it is 
\emph{sufficient} for the components of this triple to be
$\Gamma_{\frak{S}}$\,-\,invariant. Furthermore, it follows from
the {\it Assumption} that $p_{\scriptscriptstyle
\frak{S}}\big\vert_{U_{pq}} = p_{\scriptscriptstyle
S}\big\vert_{U_{pq}}$, since by construction the natural
$\mu_3$\,-\,action on the family of lattices
$H^1(p_{\scriptscriptstyle S}^{-1}(t),\mathbb{Z})$ preserves the
action of $\Gamma(3)$. The next proposition is an immediate
consequence of these considerations.

\begin{prop}
\label{theorem:log-trans-i} Let $\widetilde{L} \in \Pic S$ be such
that $\widetilde{L}\big\vert_{p_{\scriptscriptstyle
S}^{-1}(\Delta_q)}$ is $\mu_3$\,-\,invariant. Then one defines the
triple \eqref{triple} by descending
$\widetilde{L}\big\vert_{p_{\scriptscriptstyle S}^{-1}(\Delta_q)}$
to $U_q$, restricting $\widetilde{L}$ to
$S\setminus\{p_{\scriptscriptstyle S}^{-1}(\Delta_p) \cup
p_{\scriptscriptstyle S}^{-1}(\Delta_q)\} = \frak{S}\setminus\{U_p
\cup U_q\}$ and extending over $U_p$ by
$\Gamma_{\frak{S}}$\,-\,invariance, so that this triple coincides
with $(\frak{L}\big\vert_{U_p}, \frak{L}\big\vert_{U_q},
\frak{L}\big\vert_{U_{pq}})$ for some $\frak{L} \in \Pic
\frak{S}$.
\end{prop}

\refstepcounter{equation}
\subsection{Log\,-\,transfer construction I}
\label{subsection:log-trans-2}

Take some line bundle $L \in \Pic S$ and let $\mathcal{E} :=
p_{\scriptscriptstyle S*}L$ be the pushforward of $L$ under
$p_{\scriptscriptstyle S}$.

\begin{lemma}
\label{theorem:e-is-line-b} The sheaf $\mathcal{E}$ is either
trivial (zero) $\OO_{\P^1}$\,-\,module or it is a line bundle on
$\P^1$.\footnote{~Compare with \cite[III.11, Corollary
11.2]{bpv}.}
\end{lemma}

\begin{proof}
The claim is obvious when $L = \OO_S$. Thus we may assume that $L$
is a $(-1)$\,-\,curve because these generate $\Pic S$ and
$p_{\scriptscriptstyle S*}M_1 \otimes p_{\scriptscriptstyle S*}M_2
= p_{\scriptscriptstyle S*}(M_1 \otimes M_2)$ for any
$\OO_S$\,-\,modules $M_i$ such that both $p_{\scriptscriptstyle
S*}M_i$ are non\,-\,trivial. Let $F$ be a fiber of
$p_{\scriptscriptstyle S}$. We have $h^0(F,L\big\vert_F) = 1$ when
$F$ is smooth. Now suppose that $F = F_i$ is a ``\,triangle\,''
(cf. {\ref{subsection:pha-0-1}}). We have
$h^0(F_i,L\big\vert_{F_i}) \ge 1$ because $L$ is a section of
$p_{\scriptscriptstyle S}$. Let $F_{ij} \simeq \P^1$, $1 \le j \le
3$, be the irreducible components of $F_i$. Then $L\big\vert_{Fi1}
= \OO_{\P^1}(1)$, say, and $L\big\vert_{Fij} = \OO_{\P^1}$ for $j
> 1$. In particular, any global section from
$H^0(F_i,L\big\vert_{F_i})$ induces one of $\OO_{\P^1}(1)$ on
$F_{i1}$, whereas it is (the same) constant on the other $F_{ij}$.
Now, if $s_1,s_2 \in H^0(F_i,L\big\vert_{F_i})$ are two different
sections, then their linear combination gives a linear form on
$F_{i1} = \P^1$ with two \emph{distinct} values at $F_{i1} \cap
F_{i2}$ and $F_{i1} \cap F_{i3}$, respectively, a contradiction.
We thus obtain a family over $\P^1$ of one\,-\,dimensional linear
spaces $H^0(F,L\big\vert_{F})$, which are precisely the fibres of
$\mathcal{E}$, i.\,e. $\mathcal{E}$ is a line bundle as claimed.
\end{proof}\\

\begin{remark}
\label{remark:proj-sys} Interestingly enough,
Lemma~\ref{theorem:e-is-line-b} indicates that with any line
bundle on $S$ is associated a \emph{meromorphic projective
structure} on the base $\P^1$, which brings the theory of
integrable systems into play. We will outline in
{\ref{subsection:conclusion-1}} of the Appendix how our present
geometric constructions can be put into a ``\,physical\,''
perspective in this way.
\end{remark}

Consider $\widetilde{L} := p_{\scriptscriptstyle S}^*\mathcal{E}
\in \Pic S$ (see Lemma~\ref{theorem:e-is-line-b}). Let $t$ be the
projective coordinate on $\P^1$, for which $p = 0$, $q = \infty$
and the transition function of $\mathcal{E}$ is $t^{-d}$, some $d
\in \Z \cup \{\infty\}$, so that either $\mathcal{E} \simeq
\OO_{\P^1}(d)$ or $\mathcal{E} = 0$ when $d = \infty$. Then, since
the restriction $\widetilde{L}\big\vert_{p_{\scriptscriptstyle
S}^{-1}(\Delta_p)}$ is trivial on $F_p$, it is invariant under the
action \eqref{mu-2-act}. The same applies to $q$. It follows that
$\widetilde{L}$ may be identified with a line bundle $\frak{L}$ as
in Proposition~\ref{theorem:log-trans-i}. The following property
is evident (although important):

\begin{lemma}
\label{theorem:l-prop-push} $p_{\scriptscriptstyle
S*}\widetilde{L} = p_{\scriptscriptstyle \frak{S}*}\frak{L} =
\OO_{\P^1}(d)$.
\end{lemma}

So far the specifics of situation in
{\ref{subsection:log-trans-1}} was used in an essential way. We
proceed with relating $D^b(S)$ and $D^b(\frak{S})$.

\refstepcounter{equation}
\subsection{Log\,-\,transfer construction II}
\label{subsection:log-trans-3}

Let $L$ and $\widetilde{L}$ be as in
{\ref{subsection:log-trans-2}}. Denote by $\gamma: \widetilde{L}
\lra L$ the natural (forgetful) morphism of sheaves. Choose an
affine open cover $\frak{S} = \bigcup U_i$ such that
$L\big\vert_{U_i} \simeq \widetilde{L} \big\vert_{U_i} \simeq
\OO_{U_i}$ all $i$. Then $\gamma\big\vert_{U_i}$ becomes
multiplication by some \emph{regular} function $f_i \in
\OO_{U_i}\setminus\{0\}$. Moreover, the collection
$\left\{(f_i,U_i)\right\}$ defines an \emph{effective} divisor
$C_L \subset \frak{S}$, on which the kernel of $\gamma$ is
supported. Obviously, this procedure can be reversed, so that one
reconstructs $L$ from the pair
$(\widetilde{L},C_L)$.\footnote{~This correspondence also admits
an IS interpretation --- to be discussed in the Appendix (cf.
Remark~\ref{remark:proj-sys}).} It follows that $L$ meets the
conditions of Proposition~\ref{theorem:log-trans-i} \emph{if and
only if} the divisor $C_L$ meets these conditions. A priori it
need not, however, and we will associate below with $C_L$,
\emph{canonically}, an algebraic cycle $C_L^{\text{log}}$ on $S$,
to which Proposition~\ref{theorem:log-trans-i} does apply.

Namely, replacing the preceding $U_i$ with analytic subsets (GAGA)
and employing the identification $\widetilde{L} = \frak{L}$ at the
end of {\ref{subsection:log-trans-2}}, we get a natural
$\mu_3$\,-\,action on the line bundle $\widetilde{L}$. Take
generator $\sigma \in \mu_3$ and denote its action on
$\widetilde{L}$ by $\sigma^*$. Then, in particular,
$\sigma^*\ker\,\gamma$ is a coherent subsheaf of $\widetilde{L}$,
isomorphic to $\ker\,\gamma$. We thus obtain another effective
divisor $C_L^{\sigma}$ being the support of
$\sigma^*\ker\,\gamma$. This yields the following:

\begin{prop}
\label{theorem:log-trans-ii} For the line bundle on $S$,
associated with the divisor $C_L^{\mathrm{log}} := C_L +
C_L^{\sigma} + C_L^{\sigma^2}$, the conditions of
Proposition\,-\,definition~\ref{theorem:log-trans-i} hold.
\end{prop}

Now we define the line bundle $L^{\text{log}}$ via the natural
morphism $\widetilde{\gamma}: \widetilde{L} \lra L^{\text{log}}$
(``\,multiplication\,-\,by\,-\,local\,-\,equation\,'') with the
kernel supported on $C_L^{\text{log}}$. Then according to
Propositions~\ref{theorem:log-trans-i} and
\ref{theorem:log-trans-ii} the objects $L^{\text{log}}$,
$\widetilde{\gamma}$, $C_L^{\text{log}}$ (and $\widetilde{L}$ of
course) can be considered as the corresponding gadgets on the
surface $\frak{S}$.

\refstepcounter{equation}
\subsection{Exceptionality}
\label{subsection:log-trans-4}

Now we are going to show that the above correspondence $L \mapsto
L^{\text{log}}$ defines an equivalence between $D^b(S)$ and a
(full) subcategory of $D^b(\frak{S})$ in such a way that an
exceptional collection of line bundles on $S$ will remain
exceptional after $\text{log}$. Here is a key statement:

\begin{proposition}
\label{theorem:log-trans-exc} We have
$\Ext^{\bullet}(L_1^{\mathrm{log}},L_2^{\mathrm{log}}) = 0$ for
any two line bundles $L_i \in \Pic S$ satisfying
$\Ext^{\bullet}(L_1,L_2) = 0$.
\end{proposition}

\begin{proof}
Let us associate with each $L_i$ the pair
$(\widetilde{L_i},C_{L_i})$ as in {\ref{subsection:log-trans-3}}.
Then $\widetilde{L_i} = p_{\scriptscriptstyle S}^*\OO_{\P^1}(d_i)$
for some $d_i \in \Z \cup \{\infty\}$ by construction. Note also
that the cycles $C_{L_i}$ belong to the fibers of
$p_{\scriptscriptstyle S}$ and hence $p_{\scriptscriptstyle S*}
L_i = p_{\scriptscriptstyle S*} L_i^{\text{log}}$ for all $i$.

Recall next that $\Ext^{\bullet}(L_1,L_2) =
H^{\bullet}(S,L_1^*\otimes L_2) \ (= 0)$ and similarly for
$L_i^{\text{log}}$ (cf. the proof of
Theorem~\ref{theorem:exc-collection-12}). Then the vanishing of
$\Ext^{0}(L_1^{\text{log}},L_2^{\text{log}})$ follows right away
because \[H^0(\frak{S},L_1^{*\text{log}}\otimes L_2^{\text{log}})
= H^0(\frak{S},\widetilde{L^*_1}\otimes \widetilde{L_2}) =
H^0(\P^1,\OO_{\P^1}(d_2 - d_1)) = H^0(S,\widetilde{L^*_1}\otimes
\widetilde{L_2}) = H^{0}(S,L_1^*\otimes L_2) = 0\] by
Lemma~\ref{theorem:l-prop-push} and Leray spectral sequence (to
simplify the notation we use the same symbol $\widetilde{L_i}$ for
the corresponding line bundles on both $S$ and $\frak{S}$).

Further, since \[H^{0}(S,L_1^*\otimes L_2) = 0 = \Ext^{2}(L_1,L_2)
= H^0(S,\omega_{\scriptscriptstyle S}\otimes L_1\otimes L^*_2)
\]
by assumption and Serre duality, and $p_{\scriptscriptstyle
S*}\omega_{\scriptscriptstyle S} = \OO_{\P^1}(-1)$ (see the
construction of $S$ in {\ref{subsection:pha-0-1}}), it follows
that $|d_1 - d_2| = \infty$ (i.\,e. $p_{\scriptscriptstyle
S*}(L_1\otimes L^*_2) = p_{\scriptscriptstyle S*}(L^*_1\otimes
L_2) = 0$ in the setting at the end of
{\ref{subsection:log-trans-2}}). Then, similarly as earlier, we
obtain \[\Ext^{2}(L_1^{\text{log}},L_2^{\text{log}}) =
H^0(\frak{S},\omega_{\scriptscriptstyle \frak{S}}\otimes
L_1^{\text{log}}\otimes L_2^{*\text{log}}) = H^0(\frak{S},0) = 0\]
because $p_{\scriptscriptstyle
\frak{S}*}\omega_{\scriptscriptstyle \frak{S}} = 0$ by
\eqref{can-bun} and definition of $p_{\scriptscriptstyle
\frak{S}*}$ (cf. the proof of Lemma~\ref{theorem:e-is-line-b}).

It thus remains to prove the vanishing of
$\Ext^{1}(L_1^{\text{log}},L_2^{\text{log}}) =
H^1(\frak{S},L_1^{*\text{log}}\otimes L_2^{\text{log}})$. Note
that \beq\nonumber H^1(\frak{S},L_1^{*\text{log}}\otimes
L_2^{\text{log}}) = H^0(\P^1,R^1p_{\scriptscriptstyle
\frak{S}*}(L_1^{*\text{log}}\otimes L_2^{\text{log}})) \eeq by
Leray and equality $|d_1 - d_2| = \infty$. Using the latter
condition once again, by relative Serre duality \cite[III.12]{bpv}
we get
\[ H^0(\P^1,R^1p_{\scriptscriptstyle
\frak{S}*}(L_1^{*\text{log}}\otimes L_2^{\text{log}})) =
H^1(\P^1,\omega_{\scriptscriptstyle \frak{S}\slash\P^1} \otimes
p_{\scriptscriptstyle \frak{S}*}(L_1^{\text{log}}\otimes
L_2^{*\text{log}})) = 0, \] where $\omega_{\scriptscriptstyle
\frak{S}\slash\P^1} := \omega_{\scriptscriptstyle \frak{S}}
\otimes p_{\scriptscriptstyle \frak{S}}^*\OO_{\P^1}(2)$ is the
\emph{dualizing sheaf} of $p_{\scriptscriptstyle \frak{S}}$. This
completes the proof.
\end{proof}\\

\begin{remark}
\label{remark:fjfjf-kffk} The proof of
Proposition~\ref{theorem:log-trans-exc} actually shows that
$\Ext^{1}(L_1,L_2) = 0$ for any $L_i \in \Pic S$ satisfying
$\Ext^{0}(L_1,L_2) = \Ext^{2}(L_1,L_2) = 0$. Of course, this is
due to the special geometric situation dictated by the surface
$S$, and it need not hold for arbitrary rational surfaces. One may
call a collection of exceptional objects $A_1,\ldots,A_N$
\emph{quasi\,-\,exceptional} if the conditions, similar to those
in {\ref{subsection:pre-3}}, are satisfied, but the vanishing of
$\Ext^{1}(A_i,A_j)$ is not required. It would be interesting to
clarify the categorical meaning of quasi\,-\,exceptional
collections.
\end{remark}

From Proposition~\ref{theorem:log-trans-exc} we deduce the main
result of this section:

\begin{corollary}
\label{theorem:log-trans-exc-cor} Let $A_1,\ldots,A_{12}$ be an
exceptional collection of maximal length of line bundles on $S$
(cf. Remark~\ref{remark:exc-num-exc} and \cite{orlov}). Then
$A_1^{\mathrm{log}},\ldots,A_{12}^{\mathrm{log}}$ is an
exceptional collection of maximal length on $\frak{S}$.
\end{corollary}

\begin{proof}
Fix some $i$ and put $L := A_i$. Then, in the notations of
{\ref{subsection:log-trans-2}}, it suffices to prove that
$\widetilde{L} \ne 0$ (aka $d \ne \infty$). This is evident,
however, by definition of $\gamma: \widetilde{L} \lra L$ from
{\ref{subsection:log-trans-3}}.
\end{proof}


\bigskip

\renewcommand{\thesubsection}{\bf A.\arabic{equation}}

\renewcommand{\theequation}{A.\arabic{equation}}

\renewcommand{\thesection}{}

\setcounter{equation}{0}

\section*{Appendix}


\refstepcounter{equation}
\subsection{Local systems}
\label{subsection:conclusion-1}

We suggest the reader to consult e.\,g. \cite{all-bri} and
\cite{ovs-tab} for the notions and facts used in this subsection
(the literature on the subject is quite vast and we had to limit
ourselves while referring).

Recall that the fibration $p_{\scriptscriptstyle S}: S \lra \P^1
=: \frak{B}$ is \emph{modular}, with $\frak{B}$ being the
$\Gamma(3)$\,-\,curve, so that there is a natural (quotient)
morphism $\frak{B} \lra \P^1$ of degree $12$ ramified \emph{only}
at $\{0,1,\infty\}$. This yields a triangulation $\displaystyle
\bigcup_{i=1}^{24} \overline{\Delta_i} = \frak{B}$, symmetric with
respect to the antipodal involution, and a \emph{meromorphic
projective structure} subject to this triangulation, given by the
sheaf $p_{\scriptscriptstyle S*}\Z_S$. On the open locus
$\displaystyle \bigcup_{i=1}^{24} \Delta_i$, over which
$p_{\scriptscriptstyle S}$ is smooth, the latter is given by
trivializations of the family $H^1(p_{\scriptscriptstyle
S}^{-1}(t),\mathbb{Z})$ from {\ref{subsection:log-trans-1}} over
each open triangle $\Delta_i$, compatibly with the action of
$\Gamma(3)$. In addition, the so\,-\,obtained trivial rank $2$
vector bundle over each $\Delta_i$ is supplied with a flat (Gauss
-- Manin) connection, defined \emph{up to gauge}, which can be
equivalently described by the (Schr\"odinger) operator $P_i := d
\slash dz^2 + \varphi(z)\star$\ for some holomorphic quadratic
differential $\varphi(z)dz^{\otimes 2}$ and complex variable $z
\in \Delta_i$.

Further, the choice of a line bundle $L$ on $S$ yields a basis
$\{g_1,g_2\}$ in the space of solutions to $P_i(g) = 0$, $1 \le i
\le 24$. Namely, trivialization $L \simeq \C \times \Delta_i$ for
each $i$, given by some global section $s_i(z)$ of $L
\big\vert_{\Delta_i}$, yields a holomorphic family of vectors $g_1
:= s_i + \bar{s_i}$, $g_2 := \frac{1}{\sqrt{-1}}(s_i - \bar{s_i})
\in H^1(p_{\scriptscriptstyle S}^{-1}(t),\mathbb{Z}) \otimes \R$,
which are parallel with respect to the connection corresponding to
$P_i$.

One interprets the global sections of $L$ as follows. Consider the
Lie algebra $sl_2(z,z^{-1})$ over the Laurent field $\C((z))$. It
admits a non\,-\,trivial (Gelfand\,--\,Fuchs) $2$\,-\,cycle, which
gives a central extension $\V$ of $sl_2(z,z^{-1})$, the
\emph{Virasoro algebra}. The space of $P_i$ above is identified
with the coadjoint representation of $\V$. Then, for each $1 \le i
\le 24$, a given $s \in H^0(S,L)$ determines the coadjoint orbit
$O_i := \V \cdot P_i$ in such a way that

\begin{itemize}

    \item for any $i \ne j$, there is $g_{ij}\in\Gamma(3)$ with $O_i =
    g_{ij}(O_j) := \V \cdot g_{ij}^*P_j$,

    \smallskip

    \item $g_{ii} = 1$ for all $i$,

    \smallskip

    \item for any $i$, $j$, $k$, we have $g_{ij}g_{jk}g_{ki} = 1$.

\end{itemize}
This follows from the fact that the $\V$\,-\,orbits of $P_i$ are
determined by monodromy of the corresponding second order
differential equations (note once again that $P_i$ a priori depend
on $L$ --- aka on $s_i$ above). The collection $\{O_i,g_{ij}\}$
defines another meromorphic projective structure on $\frak{B}$
(with connection in the same gauge class as the initial one).

\begin{example}
\label{example:standard-proj-str} Let $j: \frak{B} \lra \P^1$ be
the $j$\,-\,map of $p_{\scriptscriptstyle S}$ of degree $12$. Then
the corresponding projective structure is given by $\varphi :=
\displaystyle\big(\frac{j''}{j'}\big)' +
\displaystyle\frac{1}{2}\big(\frac{j''}{j'}\big)^2$, the
\emph{Schwartz derivative} of $j$, with $L$ being one of the
sections of $p_{\scriptscriptstyle S}$. This projective structure
is nothing but the one associated with the homological invariant
of the elliptic fibration on $S$ (cf. \cite[III.11]{bpv} or the
proof of Lemma 3.9 in \cite{bo-pet-yu}).\footnote{~All this goes back probably to the elliptic functions approach in the
geometry of icosahedron in \cite{klein}.}
\end{example}

\begin{remark}
\label{remark:spectral-curve} The preceding arguments intimately
relate our algebro\,-\,geometric considerations with ``\,Lax
dynamics\,'' $\dot{X} = [M,X]$, $M \in \V$, in vector bundles on
$\P^1$. It would be interesting to obtain some integrable systems
in this way (cf. \cite{viniti-1}, \cite{viniti-2}, \cite{veselov},
\cite{RSS}). Furthermore, one may interpret the pair
$(\widetilde{L},C_L)$ in {\ref{subsection:log-trans-3}} as a
``\,spectral data\,'' associated with $L$, that is a (spectral)
curve plus a line bundle on it. This brings in a relation of our
present setting with Nahm, KdV and KZ equations, on the one hand
(see \cite{hitchin}, \cite{mumford}, \cite{tyurin}), and with
Painlev\'e equations on the other (see \cite{kajiwara-et-al}).
Namely, the change of projective structure on $\P^1$ given by $L$
and $s$ as above should correspond, dually, to evolution of
$(\widetilde{L},C_L)$ in the ``\,phase space\,'' $S$. In this
regard one may identify generic fiber $F_{\zeta}$ of
$p_{\scriptscriptstyle S}$ with (normalization of) $C_L$ and
represent $S$ in terms of a pair of commuting differential
operators (cf. \cite{burban-zheglov}). It would be also
interesting to relate this picture with the \emph{quantum
Calogero\,--\,Painlev\'e correspondence} between Painlev\'e and
non\,-\,stationary Schr\"odinger equations (see e.\,g.
\cite[Section 7]{zotov-smirnov}).
\end{remark}

\refstepcounter{equation}
\subsection{Representation theory}
\label{subsection:conclusion-2}

Constructions in {\ref{subsection:log-trans-2}} and
{\ref{subsection:log-trans-3}} are related to the following
modification of the projective structure on $\frak{B}$ associated
with a given $L \in \Pic S$. Namely, let $p(z)$ be a complex
polynomial, with $\deg p = 3$ and $\mult_0 p = 2$. Consider the
corresponding morphism $p: \P^1 \lra \frak{B}$ and the new projective
structure on $\frak{B}$, given by the pullbacks
$p^{-1}(\Delta_i)$, $p^*P_i$ for all $i$ (cf.
{\ref{subsection:conclusion-1}}). By construction, the monodromy
$\Gamma$ of this new structure is a non\,-\,trivial extension of
$\mu_3$ by $\Gamma(3)$, as in {\ref{subsection:log-trans-1}}. It
is then tempting to propose the following:

\begin{conj}
\label{conjecture:log-cor} In the previous setting, projective
structure on $\P^1$ associated with $L$ and $p$ coincides with the
one given by $L^{\mathrm{log}}$ on $\frak{S}$, so that $\Gamma =
\Gamma_{\frak{S}}$.
\end{conj}

Our point is that co\,-\,adjoint representations of $\Gamma(3)$
(resp. of $\Gamma$), together with a triangulation of $\P^1$, should
determine $D^b(S)$ (resp. $D^b(\frak{S})$). More precisely, the
``\,fusion matrices\,'' $g_{ij}$ in
{\ref{subsection:conclusion-1}} define the ``\,conformal
blocks\,'' $s$, from which one reconstructs $L$. The picture here
is quite similar to TQFT (cf. \cite{tyurin}). In turn, applying
$p^*$ as above may be regarded as a ``\,t'Hooft transform\,'' (or
a ``\,change of magnetic charge\,'') for initial integrable system
associated with $L$ (cf. Remark~\ref{remark:spectral-curve}). This
corresponds to the change of a $\V$\,-\,$\Gamma(3)$\,-\,bimodule
structure to a $\V$\,-\,$\Gamma$\,-\,bimodule structure on the
space of quadratic differentials $\varphi(z)$.
Conjecture~\ref{conjecture:log-cor} then claims that the latter
bimodule admits a \emph{geometric} realization as
$L^{\mathrm{log}}$ on the surface $\frak{S}$. One could also
verify the {\it Assumption} from {\ref{subsection:log-trans-1}} in
this way.

\begin{remark}
\label{remark:h-s} Preceding considerations may be compared
with (geometric) construction in \cite{hulek-schuett} of certain
Enriques surfaces. Namely, one starts with \emph{quadratic} $p$
and the base change as above, which yields a $\text{K3}$ surface
$\widetilde{S} \lra S$ (in our case $\widetilde{S}$ will be some
elliptic surface of Kodaira dimension $\ge 0$). The corresponding
Enriques surface is then obtained as the quotient of
$\widetilde{S}$ by an involution that comes from the Galois one
composed with the fiberwise shift by some $2$\,-\,torsion section
(such a group action is missing in our situation).
\end{remark}

Two comparisons are in order. First, one may interpret $S$ (resp.
$\frak{S}$) as a moduli space of quiver representations (cf.
\cite{abd-okawa-ueda}, \cite{mozg}), taking values in an abelian
category of $\V$\,-\,$\Gamma(3)$\,-\,bimodules (resp. of
$\V$\,-\,$\Gamma$\,-\,bimodules). The quiver here is the dual
graph of the triangulation on $\frak{B}$ from
{\ref{subsection:conclusion-1}}. In turn, with each $L$ is
associated the bimodule $\widetilde{L}$ (``\,representation\,''),
for which the curve $C_L$ plays the role of ``\,ramification
divisor\,'' of the Virasoro algebra $\V$. Similar picture is
observed in the case of \emph{AS regular quadratic} (Clifford)
algebras (see \cite{belm-pres-van-den}, \cite{rains}). Namely, any
such $\mathcal{A}$ is given by relations, corresponding to a net
of conics in $\P^2$. The latter has the \emph{discriminant curve}
$C_{\mathcal{A}} \subset \P^2$, $\deg C_{\mathcal{A}} = 3$, which
together with a pair of line bundles on it determines
$\mathcal{A}$. Let us also indicate that $\mathcal{A}$ yields a
projective structure on $\frak{B}$ corresponding to a degree four
map $\P^1 \lra \P^1$ given by some pencil of binary quartics. We
will revisit this $\V$ vs. $\mathcal{A}$ correspondence from the
categorical point of view in {\ref{subsection:conclusion-3}}
below.

Finally, an arbitrary exceptional collection of $L$ on $S$ (resp.
on $\frak{S}$) may be regarded as a collection of ``\,conformal
blocks\,'' (cf. {\ref{subsection:conclusion-3}}), so that one
could formulate the semiorthogonality of them in terms of certain
``\,\text{Ind}\,-\,\text{Res} (Frobenius) type\,'' relation
between the corresponding $\V$\,-\,$\Gamma(3)$\,-\,bimodules
(resp. $\V$\,-\,$\Gamma$\,-\,bimodules).

\refstepcounter{equation}
\subsection{Categorification}
\label{subsection:conclusion-3}

Let us now return to the considerations of
Section~\ref{section:log-trans} and provide some evidence for why
$\mathcal{A} =
\left<A_1^{\mathrm{log}},\ldots,A_{12}^{\mathrm{log}}\right>^{\bot}
\ne 0$. To reach contradiction with $\mathcal{A} = 0$ we interpret
both $K$\,-\,groups for $S$ and $\frak{S}$ in terms of the
Grothendieck groups of the bimodules considered in
{\ref{subsection:conclusion-2}} --- at the end of the day the
latter two groups will not be isomorphic (we will write
$\overline{K}_0$ to distinguish the $K$\,-\,group of modules from
that of sheaves). The idea is that once $\mathcal{A}$ is trivial,
we get $D^b(S) \simeq D^b(\frak{S})$, and so $\overline{K}_0(S) =
\overline{K}_0(\frak{S})$ (cf.
Proposition~\ref{theorem:log-trans-exc} and
Remark~\ref{remark:exc-num-exc}).  At the same time, we argue that
$\widetilde{L}$, corresponding to the $3$\,-\,multiple fibre $F_q$
of $p_{\scriptscriptstyle \frak{S}}$, gives a non\,-\,trivial
torsion element in $\overline{K}_0(\frak{S})$, whereas
$\overline{K}_0(S) \simeq \Z^{12}$.\footnote{~This should be
compared with invariants constructed in \cite{swan} in order to
distinguish rational cyclic group quotients from the irrational ones
(all over $\Q$). See also \cite{max-yura} for further general
outlook.}

To $\widetilde{L}$ there corresponds a character $\psi$ on $\V
\times \Gamma$ of Virasoro\,-\,monodromy representation associated
with the local system on $\P^1$ defined by $\widetilde{L}$. We
view $\psi$ as an element in the group ring
$\C\left[\overline{K}_0(\frak{S})\right]$. Then the condition
$p_{\scriptscriptstyle \frak{S}*}\widetilde{L}^{\otimes 3} =
\OO_{\P^1}$ translates into $\psi^3 = 1$. At the same time, due to
a non\,-\,trivial monodromy around $F_q$, we have $\psi \ne 1$,
whence a torsion in $\overline{K}_0(\frak{S})$. This is our sketch
of why $D^b(\frak{S})$ possesses phantoms in general.

We conclude this subsection by revisiting the present approach from
the point of view of \emph{non\,-\,commutative geometry} in order
to explain the distinguished role played by fibers of an elliptic
fibration in the description of $D^b$ (cf. the above curves
$C_L$). The exposition here will follow the ``\,semicommutative\,''
approach as discussed in the wonderful work \cite{rains} --- this
aims to put our direct geometric reasoning on the general grounds
of \emph{sheaves of categories} (see below). Recall that one
starts with an elliptic curve $E$ and a pair of degree two
morphisms $p_0,p_1: E \lra \P^1$. Then out of the line bundles
$\L_i := p_i^*\OO_{\P^1}(1)$ are constructed two algebras $S_i$,
their degree two extension $R \supset S_i$ and an $R$\,-\,algebra
$\overline{H}$ generated by $\text{End}_{S_i}(R)$, which is a
certain quadratic $\Z$\,-\,algebra (``\,twisted dihedral group
algebra\,''). This $\overline{H}$ turns out to be an Azumaya
algebra with a maximal order $\overline{S}$. The algebra
$\overline{S}$ is also quadratic and $\Z$\,-\,graded.

Now, coming back to our elliptic surface $S$ (or, more generally,
to $\frak{S}$), it turns out that it is encoded by the center of
an appropriate $\overline{S}$ as above and that $D^b(S)$ is
equivalent to the bounded derived category of
finite\,-\,dimensional $\overline{S}$\,-\,modules. The generating
(simple) modules are also \emph{$\overline{H}$\,-\,bimodules} and
can be extracted from the graded pieces of $\overline{S}$. This
suggests a relation between the present constructions and those
from {\ref{subsection:conclusion-2}} in the following way. First,
arguing as in \cite[Section 3]{hard-kat}, one interprets $D^b(S)$
as a category of global sections of the sheaf of categories
associated with afore\,-\,mentioned quiver (same for
$D^b(\frak{S})$). Recall that the corresponding representations
take values in the $\V$\,-\,$\Gamma(3)$\,-\,bimodules (resp.
$\V$\,-\,$\Gamma$\,-\,bimodules). Second, using the gluing and
geometric realization technique of e.\,g. \cite[Section
5]{orlov1}, one should be able to reinterpret $D^b(S)$ (resp.
$D^b(\frak{S})$) in terms of the
\emph{$\overline{H}$\,-\,bimodules} for an appropriate
$\overline{H}$ as above. This should also give a fully faithful
embedding $D^b(S) \hookrightarrow D^b(\frak{S})$ (compare with
Corollary~\ref{theorem:log-trans-exc-cor}). Finally, the fact that
both categories are actually distinct can be proved by an explicit
computation of the $K$\,-\,groups of respective modules, similarly
as we did at the beginning of this
subsection.\footnote{~Alternatively, one may apply \cite[Lemma
4.3]{bridgeland-maciocia}, which also gives $D^b(S) \ne
D^b(\frak{S})$.}

\begin{remark}
\label{remark:yang-baxter} More analogies with the preceding
discussion are in order. First, it is a construction of
trigonometric solutions to the Yang -- Baxter equation, the
``\,classical\,'' ($R$\,-\,matrix) one in \cite{bel-dri} and an
``\,$A_{\infty}$\,'' (Fukaya -- Seidel) one in \cite{lek-poli}.
Note also a similarity between the preceding
``\,Virasoro\,-\,Clifford\,'' correspondence and the
``\,Helix\,-\,Sklyanin\,'' one (cf. \cite[Ch. 8]{nakajima}) in the
context of non\,-\,commutative deformations of $\P^n$ (see
\cite[Section 7]{hard-kat} for instance).\footnote{~A feasible
reasoning should probably come from the interpretation in
\cite{keller} of Hochschield cohomology as a Lie algebra structure
on the \emph{(graded) Picard group(s)}.} Further, various surfaces
$X_q$ in \cite[Section 5]{rains} (see also \cite[Section
8]{rains}), corresponding to the above triples $(E,p_0,p_1)$ and
associated with a given rational elliptic surface $X$ (i.\,e.
$D^b(X_q)$ is obtained from $D^b(X)$ via certain gluing
bi\,-\,module), can be regarded as deformations of the
$t$\,-\,structure on $D^b(X)$ (compare with \cite[7.4]{hard-kat}).
This motivates a similar relation to hold between $D^b(\frak{S})$
and $D^b(S)$ (cf. {\ref{subsection:conclusion-4}} below). Finally,
we note that in the case of \emph{generic Enriques surface}
$\Sigma$ there is a similar interpretation for $D^b(\Sigma)$ via
$\text{Perf}(\mathcal{P},\mathcal{R})$ (see \cite[Lemma
6.14]{kuz-per}), where $\mathcal{P}$ is a stacky projective plane
and $\mathcal{R}$ is an Azumaya sheaf on it. The argument in
\cite{kuz-per} is based on the construction of $\Sigma$ as a
codimension $3$ linear section of the join of two $\P^2$s (cf.
\cite[Lemma 6.10]{kuz-per}). This suggests yet another outlook on the
categorical properties of Dolgachev surfaces and their explicit
construction (compare with Remark~\ref{remark:h-s}).
\end{remark}

\refstepcounter{equation}
\subsection{A\,-\,side}
\label{subsection:conclusion-4}












This final subsection discusses some analogies with the Seiberg --
Witten theory on $4$\,-\,manifolds. Here \emph{moduli spaces of
instantons} on $S$ will play the role of geometric counterparts
for $D^b(S)$. Let us recall the relevant notions.

First of all, if $M$ is the $4$\,-\,manifold underlying $S$, then
$K_S$ corresponds to a $\text{spin}^c$\,-\,structure on $M$ (same
also applies to $\frak{S}$ being homeomorphic to $M$). This yields
a decomposition of the complexified tangent bundle $T_S \otimes \C
= (W^{-})^* \otimes W^{+}$ and the (gauge class of) Dirac operator
$D_S: H^0(S,W^{+}) \to H^0(S,W^{-})$. Here $W^{\pm}$ are rank two
vector bundles satisfying $\bigwedge^2 W^{\pm} = -K_S$. Fixing a
rank $2$ vector bundle $\mathcal{V}$ on $S$ with Chern classes
$c_i(\mathcal{V})$ one obtains a moduli space of twisted Dirac
operators $D_S^a: H^0(S,\mathcal{V} \otimes W^{+}) \to
H^0(S,\mathcal{V} \otimes W^{-})$ for various Hermitian
connections $a$ on $\mathcal{V}$. Imposing further restrictions on
the solutions of the equations $D_S^a \star = 0$, as well as on
the curvature of $a$, yields various moduli spaces of
\emph{topological instantons} whose geometric characteristics
provide Donaldson, SW, etc. invariants of $M$, specific to the
smooth structure induced from $S$.

\begin{example}[{cf. \cite[Ch. 1, \S 2]{tyurin-spin-pol}}]
\label{example:tyurin} Consider a collection of vectors $e_i \in
\text{Pic}\,S$ satisfying: (i) $e_i \equiv K_S \mod 2$; (ii)
$K_S^2 - 4c_2(S) \le e_i^2 \le 0$; (iii) $e_i \cdot K_S = 0$. The
hyperplanes (``\,walls\,'') $e_i^{\bot}$ divide the space $\Pic S \otimes
\R$ into a system of chambers $\mathcal{C}_k$. Then we can form
the moduli space $J_k^l$ of Dirac operators, corresponding to
vector bundles $\mathcal{V}$ with $c_1(\mathcal{V}) = c_1(S)$ and
$c_2(\mathcal{V}) = c_2(S) + k$ as above, which are also
semistable with respect to polarization belonging to the
$k^{\text{th}}$ chamber $\mathcal{C}_k$. Imposing further a
\emph{quasi\,-\,parabolic structure} on $\mathcal{V}$ with respect to
some curve $C \subset S$, one gets a family of moduli spaces
$J_{k,\alpha}^l$ of $\alpha$\,-\,parabolic vector bundles, $\alpha
\in [0,1/2]$. These $J_{k,\alpha}^l$ also turn out to be
birationally isomorphic.
\end{example}

Going back to the categorical picture let us say first that the
choice of a spin structure and topological data as above parallels
that of a numerically exceptional collection (cf.
Remark~\ref{remark:exc-num-exc}). Further, the exceptionality
condition for collections of objects in $D^b$ may be considered as
an analog of Dirac or Seiberg -- Witten equations, which brings in
a similarity between $D^b$ and various ``\,Jacobians\,'' $J_k^l$
from Example~\ref{example:tyurin}. We expect the properties of
the latter might guide the properties of the former. Of most
interest to us is the behavior under the logarithmic
transformations (cf. Section~\ref{section:log-trans}).

Namely, we propose that the ``\,core\,'' $D^b(S)$ is preserved
under these transformations (cf. the discussion at the end of
{\eqref{subsection:conclusion-3}}), which is reminiscent of the
above birationality of $J_{k,\alpha}^l$.

\begin{remark}
\label{remark:stability-conditions} Representing one log transform
on a minimal elliptic surface as a \emph{rational blow\,-\,down}
one finds that the SW invariants are preserved (see e.\,g.
\cite[Lecture 4]{FS}). On the other hand, performing such
transform on $S$ yields an \emph{algebraic} rational elliptic
surface $\Sigma$, with $D^b(\Sigma)$ being a deformation of
$D^b(S)$ (see \cite{AKO}). This should give yet another hint for
the behavior of $D^b$ under the logarithmic transformations of
elliptic surfaces in general. Namely, we expect that the canonical
class $K_S$ defines a \emph{stability condition} on $D^b(S)$
w.\,r.\,t. $H^2(S,\mathbb{Z}) = H^2(\Sigma,\mathbb{Z})$ (cf.
\cite[Section 5]{mozg-stab}), which deforms to another stability
condition given by $K_{\Sigma}$. This deformation procedure should
in turn preserve (essential) parts of the categories $D^b(S)$ and $D^b(\Sigma)$.
\end{remark}

\bigskip

\thanks{{\bf Acknowledgments.} The work of the first author was performed at the Center of Pure Mathematics, MIPT, with
financial support of the project FSMG\,-\,2023\,-\,0013. Both authors were supported by the Basic Research Program of the
National Research University Higher School of Economics (HSE-BR-2025-023).

\end{document}